\documentclass[11pt]{amsart}
\usepackage[a4paper, total={6.5in, 9in}]{geometry}
\usepackage{amsthm}
\usepackage{amsmath}
\usepackage{amsxtra}
\usepackage{amscd}
\usepackage{amssymb}
\usepackage{mathrsfs}
\usepackage[table]{xcolor}
\usepackage{mathtools}
\usepackage{xspace}
\usepackage{amsfonts}
\usepackage{mathrsfs}
\usepackage{todonotes}
\usepackage{enumerate}
\usepackage{multirow}
\usepackage{tikz}
\usetikzlibrary{tikzmark}
\usepackage{hyperref}
\usepackage{cancel}
\usepackage{tikz-cd}
\usepackage{graphicx}
\usepackage[all]{xy}
\usepackage{amscd}
\usepackage{caption}
\usepackage{epsfig}
\usepackage{graphicx,transparent}
\usepackage{here}
\usepackage{hyperref}
\usepackage{setspace}
\usepackage{wasysym}
\usepackage{xparse}
\usepackage{manfnt}

\usetikzlibrary{calc}
\usepackage{listings}
\usepackage{textcomp}
\usepackage{microtype}
\makeatletter\g@addto@macro\@verbatim{\microtypesetup{activate=false}}\makeatother

\definecolor{mygreen}{rgb}{0,0.6,0}
\definecolor{mygray}{rgb}{0.5,0.5,0.5}
\definecolor{mymauve}{rgb}{0.58,0,0.82}
\definecolor{backcolour}{rgb}{0.95,0.95,0.92}

\usepackage[
            backend=biber,
            style=alphabetic,
            maxalphanames=6,
            maxnames=7
            ]{biblatex}
\bibliography{reference}

\input{pre.tex}

\title{Rational cohomology of $\mathcal{M}_{4,1}$}

\author{Yiu Man Wong}
\address{Institut f\"ur Mathematik \\ Goethe-Universit\"at Frankfurt \\ 60325 Frankfurt am Main, Germany}
\email{wong@math.uni-frankfurt.de}

\author{Angelina Zheng}
\address{Department of Mathematics and Physics \\ University of Roma Tre \\ 00146 Rome, Italy}
\email{angelina.zheng@uniroma3.it}

\begin{document}

\begin{abstract}
    We compute the rational cohomology of the moduli space $\mathcal{M}_{4,1}$ of non-singular genus $4$ curves with $1$ marked point, using Gorinov-Vassiliev's method. 
\end{abstract}

\maketitle

\section{Introduction}
Let $\mathcal{M}_{g,n}$ be the moduli space of non-singular genus $g$ curves with $n$ marked point. The rational cohomology of $\mathcal{M}_{g,n}$ has been investigated in the following cases:
\begin{itemize}
    \item $\mathcal{M}_{0,n}$: the ring structure of its rational cohomology is given in \cite{Get95}, while multiplicities of irreducible $\mathfrak{S}_n$-representations are determined in \cite{KL};
    \item $\mathcal{M}_{1,n}$ for $n\leq 9$: the $\mathfrak{S}_n$-equivariant Hodge Euler characteristic is computed in \cite{Ber08};
    \item $\mathcal{M}_{2,n}$ for $n\leq 7$: the $\mathfrak{S}_n$-equivariant Hodge Euler characteristic is computed in \cite{Ber09};
    \item $\mathcal{M}_{3,n}$: the rational cohomology ring is known for $n=0,1,2$ by \cite{Loo93}, \cite{BT},\cite{Tom07}, respectively, and for $3\leq n\leq 5$: the $\mathfrak{S}_n$-equivariant Hodge Euler characteristic is computed in \cite{Ber09};
    \item $\mathcal{M}_{4,n}$: the rational cohomology ring is known for $n=0$ by \cite{Tom} and the $\mathfrak{S}_n$-equivariant Hodge Euler characteristic is computed in \cite{BFP} for $n\leq 3$.
\end{itemize}
In this paper, we will compute the rational cohomology of $\mathcal{M}_{4,1}.$ We will follow the strategy in \cite{Tom} and \cite[Section 5]{BT}. The result will be expressed in terms of the \emph{Hodge-Grothendieck polynomial}. Let $K_0(\mathsf{MHS}_{\mathbf{Q}})$ denote the Grothendieck group of mixed Hodge structures and write $\mathbf{L}$ for the class of the Tate-Hodge structure $\mathbf{Q}(-1)$ of weight $2$. 
    
\begin{dfx}
We define the Hodge-Grothendieck polynomials of a graded $\mathbf Q$-vector space $T_\bullet$ which carries a mixed Hodge structure as 
$$P_{T_{\bullet}}(t):=\sum_{i \geq 0}\left[T^i\right]t^i\in K_0(\mathsf{MHS}_{\mathbf{Q}})\left[t\right].$$
\end{dfx}

If $X$ is a quasi-projective variety, then its rational cohomology and its Borel-Moore homology carry a mixed Hodge structure and we will denote by 
    $${P}_X(t):=\sum_{i \geq 0}\left[H^i(X;\mathbf{Q})\right]t^i\in K_0(\mathsf{MHS}_{\mathbf{Q}})\left[t\right],$$
$$\oline{P}_X(t):=\sum_{i \geq 0}\left[\oline{H}_i(X;\mathbf{Q})\right]t^i\in K_0(\mathsf{MHS}_{\mathbf{Q}})\left[t\right],$$ their Hodge-Grothendieck polynomials, respectively.

We will work with mixed Hodge structures which are sums of Tate-Hodge structures. Therefore, the Hodge-Grothendieck polynomials appearing in this paper are equivalent to the cohomology or the Borel-Moore-homology with rational coefficients.

\begin{thm}\label{thm:Main_thm}
    The Hodge-Grothendieck polynomial of $\mathcal{M}_{4,1}$ is
    $$P_{\mathcal{M}_{4,1}}(t)= 1 + 2\mathbf{L}t^2 + 2\mathbf{L}^2t^4 + 2\mathbf{L}^3 t^5 + \mathbf{L}^3 t^6 + \mathbf{L}^4 t^7+ \mathbf{L}^8t^9.
    $$
\end{thm}

This agrees with the polynomial point count of $\mathcal{M}_{4,1}$ in \cite[Theorem 1.5]{BFP}, which is $\#\mathcal{M}_{4,1}(\mathbf F_q) =q^{10}+2q^9+2q^8-q^7-q^6-q^2$. Indeed, from \cite[Theorem 2.1.8]{HRV}, replacing $q$ in $\#\mathcal{M}_{4,1}(\mathbf F_q)$ by $\mathbf L$ gives the (compactly supported) Hodge Euler characteristic. Then by Poincar\'{e} duality, one can compare the results.

We recall that a non-hyperelliptic genus $4$ curve can be canonically embedded in $\mathbf{P}^3$ as the complete intersection of a cubic surface with an irreducible quadric surface. The irreducible quadric surface can either be a non-singular quadric surface or a quadric cone. Thus we can stratifiy the space $\mathcal{M}_{4,1}$ as follows:
\begin{align}\label{strat}
   C_{2,1}\subset \overline{C}_{1,1} \subset \overline{C}_{0,1} = \mathcal{M}_{4,1}, 
\end{align}
where 
\begin{itemize}
    \item $C_{0,1}$ is the locus of non-hyperelliptic genus $4$ curves (with one marked point) that are complete intersections with non-singular quadric surfaces; 
    \item $C_{1,1}$ is the locus of non-hyperelliptic genus $4$ curves (with one marked point) that are complete intersections with quadric cones;
    \item $C_{2,1}$ is the locus of hyperelliptic genus $4$ curves (with one marked point).
\end{itemize}

The rational cohomology of the closed stratum is well known.

\begin{thm}\label{thm:Hyp} For any $g\geq2,$ there is an isomorphism (of mixed Hodge structures) $H^\bullet(\mathcal{H}_{g,1};\mathbf Q)\cong H^\bullet(\mathbf P^1;\mathbf Q),$
where $\mathcal{H}_{g,1}$ is the moduli space of hyperelliptic genus $g$ curves with a marked point.
\end{thm} \begin{proof}
Consider the Leray spectral sequence of $\mathcal H_{g,1}\to \mathcal H_g.$
Since the base space has the rational cohomology of a point, the rational cohomology of $\mathcal H_{g,1}$ is equal to the invariant part, with respect to the hyperelliptic involution, of the cohomology of a hyperelliptic curve.
\end{proof}

The space $C_{0,1}$
parametrizes non-singular $(3,3)$-curves (with one marked point) on $\mathbf{P}^1\times\mathbf{P}^1,$
while the space $C_{1,1}$ 
parametrizes non-singular degree $6$ curves (with one marked point) on the weighted projective plane $\mathbf{P}(1,1,2).$ 

Following \cite[Section 5]{BT}, we consider the incidence correspondence 
$$\mathcal{I}_i:=\{(f,p)\in X_i\times Z_i: f(p)=0\},$$
for $i=0,1.$
Here $Z_0=\mathbf{P}^1\times\mathbf{P}^1$,
$Z_1=\mathbf{P}(1,1,2)\setminus \{[0:0:1]\}$ and $X_i$ is the open space parametrizing non-singular (weighted) polynomials in $V_i$ with $V_0=\Gamma(\mathcal O_{\mathbf{P}^1}(3)\otimes O_{\mathbf{P}^1}(3))$ and $V_1=\Gamma(\mathcal O_{\mathbf{P}(1,1,2)}(6)).$

Then $C_{i,1}$ can be regarded as the quotient $\left[\mathcal I_i /G_i\right],$ with $G_i=\operatorname{Aut}V_i,$ and in particular they have same rational cohomology. 

The space $\mathcal{I}_i$ has a natural forgetful map $\pi_i:\mathcal{I}_i\rightarrow Z_i,$ that is a locally trivial fibration.
We will compute the cohomology of the fibre $X_{i}^{p}$ using the Gorinov-Vassiliev's method, \cite {Vas}, \cite{Gor}, \cite{Tom}, and then deduce the cohomology of $\mathcal{I}_i$ via the Leray spectral sequence associated to $\pi_i.$

Finally, the rational cohomology of the quotients can be obtained via the generalized Leray-Hirsch theorem from \cite[Theorem 3]{PS}. 

Gorinov-Vassiliev's method and the generalized version of Leray-Hirsch theorem will be briefly recalled in Section 2, together with some homological lemmas which will be used throughout the work. 
The cohomology of $C_{0,1}, C_{1,1}$ will be computed in Sections 3 and 4, respectively.

\begin{thm}\label{thm_C_0}
    The Hodge-Grothendieck polynomial of $C_{0,1}$ is
    $$P_{C_{0,1}}(t)= 1 + \mathbf{L}t^2 + \mathbf{L}^3 t^5  + \mathbf{L}^4 t^7+ \mathbf{L}^8t^9.
    $$
\end{thm}

\begin{thm}\label{thm_C_1}
     The Hodge-Grothendieck polynomial of $C_{1,1}$ is
    $$P_{C_{1,1}}(t)= 1 + \mathbf{L}t^2 + \mathbf{L}^2 t^3.
    $$
\end{thm}

Finally, the Gysin spectral sequence associated to the stratification \ref{strat}, whose columns are determined by Theorems \ref{thm:Hyp}, \ref{thm_C_0}, \ref{thm_C_1}, yields the main result in Theorem \ref{thm:Main_thm}.

\begin{proof}[Proof of Theorem \ref{thm:Main_thm}]

The only classes in the spectral sequence from \ref{strat}, having same weight for which a non-trivial differential might exist are those supported in degrees $5,6$. Such a differential is trivial because of the non-vanishing of $H^6(\mathcal M_{4,1};\mathbf Q).$ In fact, we have the following commutative diagram.
 \begin{center}
\begin{tikzcd}
\oline{H}_{14}(\mathcal{H}_{4,1}) \arrow[r,"i_*"] \arrow[d,"\pi_*"] &\oline{H}_{14}(\mathcal{M}_{4,1})\arrow[d,"\pi_*"]\\
\oline{H}_{14}(\mathcal{H}_{4})\arrow[r,"i_*"] &\oline{H}_{14}(\mathcal{M}_{4})
\end{tikzcd}
\end{center}
Here $i$ denotes the two natural inclusions and $\phi$ is the forgetful map. 

The generator of $\oline{H}_{14}(\mathcal{H}_{4,1};\mathbf Q)$ is the class of the locus $A$ where the marked point is a Weierstrass point, i.e. $\pi_*[A] = 10[\mathcal{H}_4] \in \oline{H}_{14}(\mathcal{H}_{4})\cong H^0(\mathcal{H}_{4};\mathbf Q)$. Since $i_*[\mathcal{H}_4]\neq 0\in \oline{H}_{14}(\mathcal{M}_{4})$ (c.f. \cite{Tom}), the commutativity of the diagram yields that $i_*[A]\neq 0 \in \oline{H}_{14}(\mathcal{M}_{4,1};\mathbf Q)\cong{H}^6(\mathcal{M}_{4,1};\mathbf Q).$
\end{proof}

\begin{rem}\label{rm:SC}
From \cite{L95}, \cite{BSZ}, we already knew that $R^{3}(\mathcal{M}_{4,1})\neq 0,$ where $R^\bullet (\mathcal{M}_{4,1})$ denotes the the tautological (Chow) ring, which is equal to the rational Chow ring $A^\bullet(M_{4,1})$, \cite{CL}.

Furthermore, the cycle class map $\mu:A^\bullet(\mathcal{M}_{4,1})\to H^{2\bullet}(\mathcal{M}_{4,1};\mathbf Q)$ is an isomorphism. 
The injectivity be proved by using the \emph{CKgP property}, \cite[Definition 2.5]{BS}, which holds for $\mathcal M_{4,1}$, $\overline{\mathcal M}_{4,1}$ and any component of the boundary $\partial \mathcal M_{4,1}=\overline{\mathcal M}_{4,1}\setminus \mathcal M_{4,1}$, \cite[Section 4]{CL}. Let $\widetilde{\partial \mathcal M}_{4,1}$ denote the normalization of the boundary, we have a commutative diagram between two right exact sequences,
\begin{center}\begin{equation*}\label{exc}
    \begin{tikzcd}
        A^{i}(\widetilde{\partial \mathcal M}_{4,1})\arrow[r]\arrow[d] &A^{i}(\overline{ \mathcal M}_{4,1})\arrow[r,]\arrow[d,]&A^i(\mathcal M_{4,1})\arrow[r]\arrow[d] &0\\
        H^{2i}(\widetilde{\partial \mathcal M}_{4,1};
        \mathbf{Q})\arrow[r] & H^{2i}(\overline{ \mathcal M}_{4,1};\mathbf{Q})\arrow[r]&W_{2i}H^{2i}(\mathcal M_{4,1};\mathbf{Q})\arrow[r] &0\\
    \end{tikzcd}\end{equation*}
\end{center}
where the vertical maps are the cycle class maps. By \cite[Lemma 3.11]{CL} and \cite[Lemma 4.3]{CLP} the first two are isomorphisms and the third one is surjective, and therefore also an isomorphism. 
The surjectivity of the cycle class map then follows from Theorem \ref{thm:Main_thm} since all even degree cohomology is pure.

This can also be verified, using the \texttt{Sage} package \texttt{admcycles}. The Pixton's relations span the complete set of relations in the tautological rings $R^\bullet(\mathcal{M}_{4,1})\subset A^\bullet(\mathcal{M}_{4,1})$ and $RH^{2\bullet}(\mathcal{M}_{4,1};\mathbf Q)\subset H^{2\bullet}(\mathcal{M}_{4,1};\mathbf{Q})$,  and one can compute the ring structures.
This gives $R^\bullet(\mathcal{M}_{4,1})=RH^{2\bullet}(\mathcal{M}_{4,1};\mathbf Q)$ and comparing the ring with our result, this proves that $A^{\bullet}(\mathcal{M}_{4,1})\cong H^{2\bullet}(\mathcal{M}_{4,1};\mathbf Q)$. 
\end{rem}

\begin{rem}\label{rm:OT}
    From Theorem \ref{thm:Hyp} and \cite{BT} $H^\bullet(\mathcal{M}_{g,1};\mathbf Q)=H^\bullet(\mathcal{M}_g;\mathbf Q)\otimes H^{\bullet}(\mathbf {P}^1;\mathbf Q),$ for $g=2,3.$ Our results thus provides the first case ($g=4$) in which the rational cohomology of $\mathcal{M}_{g}$ with respect to the local system $R^1f_*\mathbf Q$ is non-trivial.

The Leray spectral sequence of the forgetful map $f:\mathcal M_{4,1}\to \mathcal M_4,$ is 
    $$E_2^{p,q}=H^p(\mathcal M_{4};R^qf_*\mathbf Q)\Rightarrow 
 H^{p+q}(\mathcal M_{4,1};\mathbf Q).$$
The rational cohomology of $\mathcal M_{4,1}$ is then determined by the cohomology of $\mathcal M_4$ with constant coefficients, and with coefficients in $R^1f_*\mathbf Q.$
The result in Theorem \ref{thm:Main_thm} together with \cite{Tom} therefore prove that the Hodge-Grothendieck polynomial of $H^\bullet(\mathcal M_4;R^1f_*\mathbf Q)$ is either $\mathbf L^3t^4+\mathbf L^8t^8$ or $\mathbf L^3t^4+\mathbf L^4t^6+\mathbf L^4t^7+\mathbf L^8t^8$.
\end{rem}

\section*{Acknowledgements}
This project started at the AGNES summer school on Intersection Theory on Moduli spaces at Brown University in 2023, we thank all the organizers and participants for the constructive environment and stimulating interactions.
We would also like to thank Samir Canning for his help with Remark \ref{rm:SC} and Orsola Tommasi for Remark \ref{rm:OT}.

The first author is supported
by the DFG-project MO 1884/2-1 and by the Collaborative Research Centre TRR 326 ``Geometry and Arithmetic of Uniformized Structures". The second author is a member of the INDAM group GNSAGA.

\section{Preliminaries}
In this section, we will briefly introduce the techniques and prove some results which we will use frequently in the following sections. 

\subsection{Gorinov-Vassiliev's method}

We briefly introduce Gorinov-Vassiliev's method, for the sake of completeness and to set the notation. For the details of the method, we refer the readers to \cite[Section 2.1]{Tom}. 

Let $Z$ be a projective variety and $V$ a vector space of global sections of some vector bundle on $Z$.
The cohomology of the open space $X\subset V$ of non-singular sections is equivalent to the Borel-Moore homology of the \emph{discriminant} $\Sigma=V\setminus X$ via Alexander duality:
\begin{equation*}
\widetilde{H}^j(X;\mathbf Q)\cong \oline{H}_{2\operatorname{dim}V-j-1}(\Sigma;\mathbf Q)\otimes \mathbf Q(-\operatorname{dim}V).
\end{equation*}
Here $\widetilde{H}^\bullet$ refers to the reduced cohomology and $\oline{H}_\bullet$ denote Borel-Moore homology, \cite[Chapter 19]{F}.

Gorinov-Vassiliev's method allows us to compute the Borel-Moore homology of a discriminant, by  constructing a simplicial
resolution for $\Sigma$, starting from a collection $X_1,\dots,X_N$ of families of configurations in $Z$, satisfying the axioms in \cite[List 2.3]{Tom}.
This collection defines a space $|\mathcal{X}|,$ together with a \emph{geometric realization} $|\mathcal{X}|\rightarrow \Sigma$ that is a homotopy equivalence and induces an isomorphism on Borel-Moore homology, which respects mixed Hodge structures.
The method provides a stratification by locally closed spaces $\{F_j\}_{j=1,\dots,n}$ of $|\mathcal{X}|.$
This gives a spectral sequence converging to $\oline{H}_\bullet(\Sigma;\mathbf Q)$ whose columns are such that
$$E^1_{p,q}=\oline{H}_{p+q}(F_p;\mathbf Q),$$
and can be computed using the following description of the strata, this spectral sequence will be called \emph{Vassiliev's spectral sequence}.
\begin{prop}[\cite{Gor}]\label{prop:Gor}
\begin{enumerate}
\item For any $i=1,\dots,N$ the stratum $F_i$ is a complex
vector bundle of rank $d_i$ over $\Phi_i$, which is a locally trivial fibration over $X_i$.
\item If $X_i$ consists of configurations of $m$ points, the fiber of $\Phi_i$ over any $K \in X_i$ is an
$(m-1)$-dimensional open simplex, which changes orientation under the homotopy
class of a loop in $X_i$ interchanging a pair of points in $K$.
\item If $X_N = \{Z\}$, $F_N$ is the open cone with vertex a point (corresponding to the
configuration $Z$), over $\cup_{j=1}^{N-1}\Phi_j.$
\end{enumerate}
\end{prop}

\subsection{Generalized Leray-Hirsch Theorem}

\begin{thm}[\cite{PS}]\label{thm:gen_leray_hirsch}

Let $X$ be an algebraic variety and $G$ be a reductive group acting on $X$ with finite stabilizers. 
Consider the orbit map 
\begin{align*}
\rho: G&\rightarrow X\\
g&\mapsto g\cdot x_0,
\end{align*}
with $x_0\in X$ fixed.

If the map induced by the orbit map 
\begin{equation*}
    \rho^*:H^\bullet(X;\mathbf Q)\rightarrow H^\bullet(G;\mathbf Q)
\end{equation*}
is surjective, then there is an isomorphism of graded linear spaces
 
\begin{align*}
	H^\bullet(X/G;\mathbf Q)\otimes H^\bullet(G;\mathbf Q)\cong H^\bullet(X;\mathbf Q).
	\end{align*}
\end{thm}

The hypotheses of Theorem \ref{thm:gen_leray_hirsch} are known to hold for moduli spaces of non-singular $(3,3)$-curves on $\mathbf P^1\times \mathbf P^1$, \cite[3.1]{Tom}, and for moduli of non-singular curves on a quadric cone, \cite[4.1]{Tom}. Moreover, by \cite[Theorem 10]{BT} this also extends to $\mathcal I_i,$ $i=0,1.$ 

By considering the action of $G_i=\Aut Z_i$ on $\mathcal{I}_i$, we then expect that \begin{equation}
\label{div}H^\bullet(\mathcal{I}_i;\mathbf Q)\cong H^\bullet(\mathcal{I}_i/G_i;\mathbf Q)\otimes H^\bullet(G_i;\mathbf{Q}). \end{equation}
 
We will make use of the divisibility of the cohomology of $\mathcal{I}_i$ to determine the ranks of the differential maps in the spectral sequence. 
\subsection{Homological lemmas}

Here we describe some spaces that will be used in further computations, together with their affine stratifications which allows us to easily compute their Borel-Moore homology, using known results.

\begin{lm}\label{lm:char_con} We have the following characterizations of spaces parametrizing non-singular conics on a non-singular quadric surface $Q$ in $\mathbf{P}^3$ (i.e. $(1,1)$-curves):

\begin{itemize}
    \item The space parametrizing non-singular conics on $Q$ is $\check{\mathbf{P}^3}\setminus \check{Q}$ is quasi-isomorphic to $\mathbf{C}^3\setminus\mathbf{C}^1.$ 
    \item The space parametrizing non-singular conics through a fixed point $p$ on $Q$ is $\check{p}\setminus (\check{\ell_1}\cup\check{\ell_2})$ for $\ell_1,\ell_2$ a line of the first and of the second ruling through $p$. The space is quasi-isomorphic to $\mathbf{C}^2\setminus\mathbf{C}^1$.
    \item The space parametrizing non-singular conics {\em not} through a fixed point $p$ on $Q$ is $\check{\mathbf{P}}^3\setminus (\check{Q}\cup \check{p})$. The space is quasi-isomorphic to $\mathbf{C}^3\setminus\mathbf{C}^2$.
\end{itemize}
\end{lm}
\begin{proof}
First, notice that a non-singular conics on $Q$ is just the intersection of a hyperplane in $\mathbf{P}^3$ not tangent to $Q$. In addition, a non-singular quadric surface in $\mathbf{P}^3$ is self-dual. Hence, the affine stratifications of the spaces involved in the lemma is given by isomorphisms:
$\check{Q}\cong \mathbf{P}^1\times\mathbf{P}^1$, $\check{\mathbf{P}}^3\cong\mathbf{P}^3$, $\check{p}\cong \mathbf{P}^2$ and $\check{\ell_1},\check{\ell_2}\cong\mathbf{P}^1.$

Since the intersections $\check{\ell_1}\cap\check{\ell_2}$ resp. $\check{Q}\cap\check{p}$ are just unions of affine strata which are $\mathbf{C}^0$ and $\mathbf{C}^1\sqcup \mathbf{C}^1\sqcup \mathbf{C}^0$ respectively, the lemma follows by removing the affine strata accordingly.
\end{proof}

We denote by $F(Z,k)$ the space of \emph{ordered} configurations of $k$ points in $Z,$ and by $B(Z,k)$ the space of \emph{unordered} ones. The local system $\pm\mathbf Q$ is locally isomorphic to $\mathbf{Q},$ and its monodromy representation equals the sign representation when interchanging a pair of points in $F(Z,k)$. 

\begin{lm}{\cite[Lemma 2]{Vas}}\label{lm_vas}
The Borel-Moore homology $\oline{H}_\bullet(B(\mathbf C^N, k);\pm \mathbf Q)$ is trivial for any $k \geq 2$.
\end{lm}

\begin{lm}[\cite{LS86}]\label{lm:FC_34}
The Borel-Moore homology of $F(\mathbf C,k),$ as $\mathfrak S_k$-representations are 
    \begin{equation*}
        \oline{H}_i(F(\mathbf C,3);\mathbf Q)=
        \begin{cases}
            \mathbf Q(3)\otimes \mathbf{S}_{3},& i=6;\\
            \mathbf Q(2)\otimes (\mathbf{S}_{3}\oplus\mathbf{S}_{2,1}),& i=5;\\
            \mathbf Q(1)\otimes \mathbf{S}_{2,1},& i=4;\\
            0&\text{otherwise,}
        \end{cases}
    \end{equation*}
    for $k=3$ and
    \begin{equation*}
        \oline{H}_i(F(\mathbf C,4);\mathbf Q)=
        \begin{cases}
            \mathbf Q(4)\otimes \mathbf{S}_{4},& i=8;\\
            \mathbf Q(3)\otimes (\mathbf{S}_{4}\oplus\mathbf{S}_{3,1}\oplus\mathbf{S}_{2,2}),& i=7;\\
            \mathbf Q(2)\otimes ({2}\mathbf{S}_{3,1}\oplus \mathbf{S}_{2,1^2}\oplus \mathbf{S}_{2,2}),& i=6;\\
            \mathbf Q(1)\otimes (\mathbf{S}_{3,1}\oplus \mathbf{S}_{2,1^2}),& i=5;\\
            0&\text{otherwise,}  
        \end{cases}
    \end{equation*}
    for $k=4.$
\end{lm}

\section{Curves on a non-singular quadric}

A non-singular quadric surface $Q$ in $\mathbf{P}^3$ is isomorphic to $Z_0=\mathbf{P}^1\times\mathbf{P}^1$ via the Segre embedding. 
Let $[x_0:x_1]$ and $[y_0:y_1]$ denote the coordinates on $\mathbf{P}^1\times\mathbf{P}^1$.
A curve on $\mathbf{P}^1\times\mathbf{P}^1$ is defined by a bihomogeneous polynomial on the variables $x_0,x_1,y_0,y_1$. 
We call the curves defined by a polynomial of bidegree $(m,n)$ an \emph{$(m,n)$-curve}. 
The lines on $\mathbf{P}^1\times\mathbf{P}^1$ are the $(0,1)$-curves, respectively $(1,0)$-curves, which we will call lines of the first, respectively second ruling.

The curves which are intersections of a fixed non-singular quadric surface with a cubic surface in $\mathbf{P}^3$ are exactly the $(3,3)$-curves on $\mathbf{P}^1\times\mathbf{P}^1$. The space of $(3,3)$-curves is a complex vector space $V=\mathbf{C}[x_0,x_1;y_0,y_1]_{(3,3)}\cong \mathbf{C}^{16}$ while the space of $(3,3)$-curves through a fixed point $p$ is a vector subspace $V_0\subset V$ of dimension $d=15$.

As in \cite[Section 3.1]{Tom}, the automorphism group $G_0=Aut(V)$ is a semi-direct product $\mathfrak{S}_2\ltimes G_0'$, where
$$G_0'=\GL(2)\times\GL(2)/\{\lambda I,\lambda^{-1}I\}_{\lambda\in\mathbf{C}^*}. $$
The action of $\mathfrak{S}_2$ is generated by the involution $\nu$ interchanging the rulings on $\mathbf{P}^1\times\mathbf{P}^1$ while the action $G_0'$ is induced by $\GL(2)$ on $\mathbf{P}^1$. 
Moreover, the cohomology of $G_0'$ is given by
\begin{align*}    H^\bullet(G_0';\mathbf{Q})\cong H^\bullet(\mathbf{C}^*;\mathbf{Q})\otimes H^\bullet(\PGL(2);\mathbf{Q})^{\otimes 2}.
\end{align*}
The geometric quotient $[\mathcal{I}_0/G_0']$ is a double cover of $C_{0,1}=[\mathcal{I}_0/G_0]$ and $H^\bullet(C_{0,1};\mathbf{Q})$ is just the invariant part of $H^\bullet(\mathcal{I}_0/G_0';\mathbf{Q})$ under the involution $\nu$.

The whole list of families of singular configurations of $(3,3)$-curves in $Z_0$ can be found in \cite[Table 2]{Tom}. We can modify that list to obtain a list of families of singular configurations of $(3,3)$-curves through $p$. Here is the list of families of singular configurations that have non-trivial contribution to $\oline{H}_\bullet(\Sigma_0;\mathbf Q)$.

\begin{table}[h!]\caption{List of singular configurations}
\begin{tabular}{c|l|c}
Type&Description&Dimension\\
\hline
1&The point $p. $&$\left[13\right]$\\
2&A point different from $p$. &$\left[12\right]$\\
3&Two points: $p$ and any other point.& $\left[10\right]$\\
4&Two points both different from $p$ and not collinear with $p$. &$\left[9\right]$\\
5&Three points: $p$ and any two other points such that they are not collinear.& $\left[7\right]$\\
6&Three points: any three points different from $p$ such that no three of them  &\\
&plus $p$ are collinear.&$\left[6\right]$\\
7&Four points: $p$ and three points such that no three of them are collinear.&$\left[4\right]$\\
8&Six points: intersection of $Q$ with three concurrent lines in $\mathbf P^3$,&\\
& which are intersections of components of three conics, containing $p.$ &$\left[1\right]$\\
9&Seven points: intersections of components of two lines of  &\\
&different rulings and two conics, containing $p$.&$\left[1\right]$\\
10&Eight points: intersection of components of the union of two pairs&\\
&of lines of different rulings and a conic, with $p$ any point on the curve.&$\left[1\right]$\\
11&The whole $\mathbf{P}^1\times \mathbf{P}^1$.&$\left[0\right]$\\
\end{tabular}
\end{table}
 \begin{table}[ht!]\caption{Spectral sequence converging to $\oline{H}_\bullet(\Sigma_0;\mathbf Q).$}\centering
\resizebox{\textwidth}{!}{
\begin{tabular}{c|ccccccccccc}
$26$&&$\mathbf{Q}(14)$&&&&&&&&&\\
$25$&$\mathbf{Q}(13)$&&&&&&&&&&\\
$24$&&$\mathbf{Q}(13)^2$&&&&&&&&&\\
$23$&&&&&&&&&&&\\
$22$&&&$\mathbf{Q}(12)$&&&&&&&&\\
$21$&&&&$\mathbf{Q}(12)^2$&&&&&&&\\
$20$&&&$\mathbf{Q}(11)^2$&&&&&&&&\\
$19$&&&&$\mathbf{Q}(11)$&&&&&&&\\
$18$&&&&&&&&&&&\\
$17$&&&&&$\mathbf{Q}(10)^2$&&&&&&\\
$16$&&&&&&$\mathbf{Q}(10)$&&&&&\\
$15$&&&&&$\mathbf{Q}(9)$&&&&&&\\
$14$&&&&&&&&&&&\\
$13$&&&&&&&&&&&\\
$12$&&&&&&&$\mathbf{Q}(8)$&$\mathbf{Q}(7)$&&&\\
$11$&&&&&&&&$\mathbf{Q}(6)^2$&$\mathbf{Q}(7)$&&\\
$10$&&&&&&&&$\mathbf{Q}(5)$&$\mathbf{Q}(6)^3$&$\mathbf{Q}(6)^4$&\\
$9$&&&&&&&&&$\mathbf{Q}(5)^3$&$\mathbf{Q}(5)^8$&\\
$8$&&&&&&&&&$\mathbf{Q}(4)$&$\mathbf{Q}(4)^4$&$\mathbf Q(6)^2+\mathbf Q(5)^4$\\
$7$&&&&&&&&&&&$\mathbf Q(5)^5+\mathbf Q(4)^8$\\
$6$&&&&&&&&&&&$\mathbf Q(4)^4+\mathbf Q(3)^4$\\
$5$&&&&&&&&&&&$\mathbf Q(3)$\\

\hline
&$(1)$&$(2)$&$(3)$&$(4)$&$(5)$&$(6)$&$(7)$&$(8)$&$(9)$&$(10)$&(11)
\end{tabular}
}
\label{Mainss_notdiv}
\end{table}

\subsection{Columns (1)--(7)}
From Proposition \ref{prop:Gor} we can immediately compute the first seven columns of the spectral sequence in Table \ref{Mainss_notdiv}.

The space $F_1$ is a $\mathbf C^{13}$-bundle over $\{p\}\cong \mathbf C^0.$

The space $F_2$ is a $\mathbf C^{14}$-bundle over $\mathbf{P}^1\times \mathbf P^1\setminus\{p\}\cong \mathbf C^2\sqcup \mathbf C^1\sqcup \mathbf C^1.$

The space $F_3$ is a $\mathbf C^{12}\times \mathring{\Delta}_1$-bundle over $\mathbf{P}^1\times \mathbf P^1\setminus\{p\}\cong \mathbf C^2\sqcup \mathbf C^1\sqcup \mathbf C^1.$

The space $F_4$ is a $\mathbf C^{10}\times \mathring{\Delta}_1$-bundle over $B(\mathbf{P}^1\times \mathbf P^1\setminus\{p\},2),$ which is quasi-isomorphic to $(\mathbf C^2 \times \mathbf C^1)^2\sqcup (\mathbf C^1\times \mathbf C^1).$

The space $F_5$ is a $\mathbf C^{7}\times \mathring{\Delta}_2$-bundle over $B(\mathbf{P}^1\times \mathbf P^1\setminus\{p\},2).$ 

The space $F_6$ is a $\mathbf C^{6}\times \mathring{\Delta}_2$-bundle over $B(\mathbf{P}^1\times \mathbf P^1\setminus\{p\},3),$ which is quasi-isomorphic to $\mathbf C^2 \times \mathbf C^1\times \mathbf C^1.$

The space $F_7$ is a $\mathbf C^{4}\times \mathring{\Delta}_3$-bundle over $B(\mathbf{P}^1\times \mathbf P^1\setminus\{p\},3).$

\subsection{Column (8)}
The configuration space $X_8$ consists of configurations of six points on the quadric surface $Q$, which are the points of intersection between three concurrent lines in $\mathbf{P}^3$ and the quadric surface $Q$. Moreover, the six points can be regarded as the intersections of three $(1,1)$-curves (conics) on $Q$ such that the fixed point $p$ lies on the union of the conics.  We have the following possibilities.
\begin{itemize}
    \item[(8a)] The fixed point $p$ is in the configuration.
    \item[(8b)] The configuration does not contain the fixed point $p$ and the conic on which $p$ lies is reducible
    \item[(8c)] The configuration does not contain the fixed point $p$ and the conic on which $p$ lies is irreducible.
\end{itemize}

\subsubsection{Configuration of type 8a}

Following \cite[3.4]{Tom}, the space $X_{8a}$ can be considered as a fiber bundle over the space $\mathbf{P}^3\setminus(Q\cup T_pQ)$,  where $q\in \mathbf{P}^3\setminus(Q\cup T_pQ)$ is the common point of the three concurrent lines and the fiber over $q$ is the configuration space parametrizing the lines.
The lines are non-tangent to $Q$ and the line $pq$ is included. 
Notice that $q$ cannot be on the tangent plane $T_pQ$ otherwise $pq$ will be tangent to $Q$. 
Any line through $q$ is a point in $\mathbf{P}(T_q\mathbf{P}^3)\cong
\mathbf{P}^2$, and the space of lines tangent to $Q$ will be a non-singular conic $C$ on $\mathbf{P}^2$. 
The line $pq$ is a point $[pq]$ on $\mathbf{P}^2\setminus C$. 

We denote by $A$ the space $\mathbf{P}^2\setminus (C\cup\{[pq]\})$. The fiber space is isomorphic to 
\[\widetilde{F}(A,2)=\left\{([\ell_1],[\ell_2])\in F(A,2): pq,\ell_1,\ell_2 \mbox{ not lying on the same plane} \right\}/\mathfrak{S}_2.\]

The orientation of the simplicial bundle over $X_{8a}$ changes sign under the action of the generator of the fundamental group of $\mathbf{P}^2\setminus C$. That means we will compute the Borel-Moore homology of $\widetilde{F}(A,2)$ with respect to the system of coefficient which is locally isomorphic to $\mathbf{Q}$ but it changes sign when a point in the configuration moves along the loop generating the fundamental group of $\mathbf{P}^2\setminus C$. We denote this local system by $W$. The space $\widetilde{F}(A,2)$ can be identified as $(\mathbf{P}^2\setminus C)^2\setminus \Delta$ where 
\[\Delta=\left\{(x_1,x_2)\in(\mathbf{P}^2\setminus C)^2: x_1,x_2,[pq] \mbox{ lie on a line }\ell\subset\mathbf{P}^2\right\}.\]

We can compute the Borel-Moore homology of $\widetilde{F}(A,2)$ by the excision sequence associated to the pair $((\mathbf{P}^2\setminus C)^2,\Delta)$. As in \cite[3.4]{Tom}, we identify the space $\mathbf{P}^2\setminus C$ with $B(\mathbf{P}^1,2)$ by considering the intersections of a line polar to $x$ with $C$ for $x\in \mathbf{P}^2\setminus C$. Then $\oline{H}_\bullet((\mathbf{P}^2\setminus C)^2;W^{\otimes 2})$ is just $\mathbf{Q}(2)$ in degree $4$ which is $\mathfrak{S}_2$-invariant. Next, we now compute the Borel-Moore homology of $\Delta$ by considering the filtration on $\Delta$ as \cite[Section 3.4]{Tom}: $$\Delta=\Delta_3\supset\Delta_2\supset \Delta_1,$$ where
\begin{align*}
    \Delta_1=&\left\{(x_1,x_2)\in(\mathbf{P}^2\setminus C)^2:x_1=x_2=[pq] \right\},\\
    \Delta_2=&\left\{(x_1,x_2)\in(\mathbf{P}^2\setminus C)^2:x_1,x_2,[pq] \mbox{ lie on a tangent to } C\right\}.
\end{align*}
By abuse of notation, we will also denote the induced system of coefficients on $\Delta,\Delta_1,\Delta_2$ by $W$. Since $\Delta_1$ is just a point, $\oline{H}_\bullet(\Delta_1;W)$ is just $\mathbf{Q}$ in degree 0 and zero in all other degrees. 

The space $\Delta_2\setminus\Delta_1$ is disjoint union of two copies of $\mathbf{C}^2\setminus\{([pq],[pq])\}$ because there are only two tangents of $C$ through $[pq]$ and $x_1,x_2$ cannot both be $[pq]$, on each of the tangents. The system of coefficients $W$ on $\Delta_2\setminus\Delta_1$ will be the constant one and hence $\oline{H}_\bullet(\Delta_2\setminus\Delta_1;W)$ is $\mathbf{Q}(2)^2$ in degree $4$, $\mathbf{Q}^2$ in degree $1$ and zero in all other degrees. 

The space $\Delta_3\setminus \Delta_2$ is a fiber bundle over $\check{[pq]}\setminus \check{C}$ with fiber isomorphic to 
$(\mathbf{C}^*)^2\setminus\{([pq],[pq])\}.$
The local system $W$ restricted to $\mathbf{C}^*$ will be the one that alternates sign along a loop around $0$, which we denote by $T$. According to \cite[Lemma 2.15]{Tom}, $\oline{H}_\bullet(\mathbf{C}^*;T)=0$ and hence $\oline{H}_\bullet((\mathbf{C}^*)^2\setminus\{([pq],[pq])\};W)$ is $\mathbf{Q}$ in degree $1$ and zero in all other degrees. Notice that $\check{[pq]}\setminus \check{C}$ is isomorphic to $\mathbf{C}^*$ and the system of coefficient we have to consider on it is the constant one. This yields that $\oline{H}_\bullet(\Delta_3\setminus \Delta_2;W)$ is $\mathbf{Q}(1)$ in degree $3$, $\mathbf{Q}$ in degree $2$ and zero in all other degrees. 

The differentials between column $1$ to $3$ of the spectral sequence in Table \ref{tb:Delta} associated with the filtration on $\Delta$ has to be of full rank due to dimensional reasons (as $V_0\setminus \Sigma_0$ is affine of dimension $15$). As a result, $\oline{H}_\bullet(\Delta;W)$ is $\mathbf{Q}(2)^2$ in degree $4$, $\mathbf{Q}(1)$ in degree $3$ and zero in all other degrees. 

\begin{table}[ht!]\caption{Spectral sequence converging to $\oline{H}_\bullet(\Delta;W)$}\label{tb:Delta}\centering
\begin{tabular}{c|ccc}
$2$&&$\mathbf Q(2)^2$&\\
$1$&&&\\
$0$&&&$\mathbf Q(1)$\\
$-1$&$\mathbf{Q}$&$\mathbf{Q}^2$&$\mathbf{Q}$\\

\hline
&$\Delta_1$&$\Delta_2$&$\Delta_3$
\end{tabular}
\label{column_10}
\end{table}

To determine the rank of the induced map by inclusion $\oline{H}_4(\Delta;W)\xrightarrow[]{}\oline{H}_4((\mathbf{P}^2\setminus C)^2;W^{\otimes 2})$, we have to consider the inclusion of one of the components of $\Delta_2\setminus\Delta_1$ into $(\mathbf{P}^2\setminus C)^2$. Such an inclusion can be factored through the inclusion of $\ell^2$ into $(\mathbf{P}^2\setminus C)^2$, where $\ell$ is a tangent to $C$. Then the rank of $\oline{H}_4(\Delta;W)\xrightarrow[]{}\oline{H}_4((\mathbf{P}^2\setminus C)^2;W^{\otimes 2})$ is $1$ by the following lemma.

\begin{lm}\label{lm:inj_incl}
    Let $\ell$ be a tangent to the conic $C$ in $\mathbf{P}^2$. Then the local system $W$ restricted to $\ell$ will be the trivial one. The pushforward of the inclusion $\ell\subset \mathbf{P}^2\setminus C$ 
 \[i_*:\oline{H}_2(\ell;W)\longrightarrow\oline{H}_2(\mathbf{P}^2\setminus C;W)\]
 is injective (and hence isomorphism).
\end{lm}
\begin{proof}
    First we know that $\oline{H}_2(\ell;W)$ and $\oline{H}_2(\mathbf{P}^2\setminus C;W)$ are $\mathbf{Q}(1)$. Since $\mathbf{P}^2\setminus (C\cup\ell)$ is affine of complex dimension two, $\oline{H}_3(\mathbf{P}^2\setminus (C\cup\ell);W)=0$. It yields then the injectivity from the long exact sequence.
\end{proof}

Hence we conclude that $\oline{H}_\bullet(\widetilde{F}(A,2);W)$ is $\mathbf{Q}(2)$ in degree $5$ and $\mathbf{Q}(1)$ in degree $4$. In addition, by Lemma~\ref{lm:char_con}, we know that $\oline{H}_\bullet(\mathbf{P}^3\setminus (Q\cup T_pQ);\mathbf{Q})$ is $\mathbf{Q}(3)$ in degree $6$, $\mathbf{Q}(2)$ in degree $5$ and zero in all other degrees. As a result, we have
\begin{equation*}
   \oline{H}_i(X_{8a};\pm\mathbf{Q})=\begin{cases}
      \mathbf{Q}(5) &i=11;\\
      \mathbf{Q}(4)^2 &i=10;\\
      \mathbf{Q}(3) &i=9;\\
      0 &\mbox{otherwise.}
   \end{cases} 
\end{equation*}

\subsubsection{Configuration of type 8b}
The space $X_{8b}$ is a fiber bundle over $\check{\ell_1}\cup \check{\ell_2}$ where $\ell_1,\ell_2$ are the two lines of the rulings through $p$. The space $\check{\ell_1}\cup \check{\ell_2}$ parametrizes reducible conics (i.e. union of two rulings) containing $p$ so it parametrizes the points $x$ on $Q$ whose tangent plane $H_x$ passes through $p$. Once we fix the reducible conic containing $p$ we can recover the configuration by picking three concurrent lines in $\mathbf{P}^3$ not through $p$. Two of them have to be on the hyperplane $H_x$ on which the chosen reducible conic lies. We denote the configuration space parametrizing the choice of the two lines by $M$. These two lines will determine a vertex $q$ and the remaining line will be chosen to be through $q$ while not lying on $H_x$. This implies that the fiber of $X_{8b}\longrightarrow\check{\ell_1}\cup \check{\ell_2}$ is fibered over $M$ with fiber isomorphic to $\mathbf{P}^2\setminus (C\cup\ell)$, where $C$ is a conic on $\mathbf{P}^2$ and $\ell$ is a tangent to $C$. This is because on the space parametrizing lines through $q$ not tangent to $Q$, the elements corresponding to lines on a certain tangent plane of $Q$ will form a line in which there is a unique element $qx$ corresponding to a line tangent to $Q$. The coefficient $\pm\mathbf{Q}$ restricted to $\mathbf{P}^2\setminus (C\cup\ell)$, which we denote by $W$, will be the one changing sign along a loop generating the fundamental group of $\mathbf{P}^2\setminus C$. Hence, by Lemma~\ref{lm:inj_incl} and the excision sequence, we conclude that $\oline{H}_\bullet(\mathbf{P}^2\setminus (C\cup\ell);W)=0$ and $\oline{H}_\bullet(X_{8b};\pm\mathbf{Q})=0$.

\subsubsection{Configuration of type 8c}
The space $X_{8c}$ is a fiber bundle over the space that parametrizes non-singular conics on $Q$ through $p$, which by Lemma~\ref{lm:char_con} is $\check{p}\setminus(\check{\ell_1}\cup\check{\ell_2})$ and 
quasi-isomorphic to $\mathbf{C}^2\setminus\mathbf{C}^1$. The fiber of the projection to a non-singular conic through $p$ is again a fiber space over the space parametrizing four points $\{z_i\}_{1=1,\dots,4}\subset C\setminus\{p\}$.

These four points determine two of the three lines of the configuration $z_1z_2$ and $z_3z_4$ and their intersection point $q=z_1z_2\cap z_3z_4.$
This space is the quotient of 
$F(\mathbf{C},4)$ by the action of the subgroup $L=\langle(12),(34),(13)(24)\rangle$ of $\mathfrak{S}_4$, which contains a normal subgroup isomorphic to $\mathfrak{S}_2\times\mathfrak{S}_2.$ We can actually view $L$ as the dihedral group $D_4$. The system of coefficients $\pm\mathbf{Q}$ restricted to $F(\mathbf{C},4)/L$ is the one changing sign along loops interchanging only the first or the second pair of points, which we will denote again by $\pm\mathbf{Q}.$ Then $\oline{H}_\bullet(F(\mathbf{C},4)/L;\pm\mathbf Q)$ can be identified with the part of $\oline{H}_\bullet(F(\mathbf{C},4);\mathbf Q)$, as $\mathfrak{S}_4$-representations, whose restriction to both factors of $\mathfrak{S}_2$ in $\mathfrak{S}_2\times\mathfrak{S}_2$ corresponds to the sign representation $\mathbf{S}_{1,1}$. Comparing the character tables of $D_4$ and $\mathfrak S_4,$ \cite{Ser}, such $\mathfrak{S}_4$-representations are $\mathbf S_{2,2}$ and $\mathbf S_{2,1^2}.$
By Lemma~\ref{lm:FC_34}, 
      \begin{equation*}     \oline{H}_i(F(\mathbf{C},4)/L;\pm\mathbf{Q})=\begin{cases}
         \mathbf{Q}(3) & i=7;\\
         \mathbf{Q}(2)^2 & i=6;\\
         \mathbf{Q}(1) & i=5;\\
         0 & \mbox{otherwise.}
     \end{cases}
    \end{equation*}

The fiber $\mathcal F$ will be the space parametrizing the remaining line, which does not lie on the hyperplane through $q$ containing the non-singular conic, therefore it is isomorphic to $\mathbf{P}^2\setminus(C\cup\ell)$ for $\ell$ not tangent to $C$.
The coefficient $\pm\mathbf{Q}$ restricted to the fiber $\mathbf{P}^2\setminus(C\cup\ell)$ will be denoted by $W$ and it will change sign along a loop generating the fundamental group of $\mathbf{P}^2\setminus C$. Since $\ell$ intersects $C$ at two points, we can consider the following long exact sequence:
\begin{align*}
    ...\longrightarrow \oline{H}_i(\mathbf{C}^*,T)\longrightarrow \oline{H}_i(\mathbf{P}^2\setminus C;W)\longrightarrow \oline{H}_i(\mathbf{P}^2\setminus (C\cup \ell);W)\longrightarrow...
\end{align*}
As $\oline{H}_\bullet(\mathbf{C}^*,T)=0$, we have $\oline{H}_\bullet(\mathbf{P}^2\setminus (C\cup \ell);W)=\oline{H}_\bullet(\mathbf{P}^2\setminus C;W)$ which is $\mathbf{Q}(1)$ in degree 2 and zero in all other degrees. By tensoring the Borel-Moore homology groups of the base and fiber spaces, we have
\begin{equation*}
   \oline{H}_i(X_{8c};\pm\mathbf{Q})=\begin{cases}
      \mathbf{Q}(6) &i=13;\\
      \mathbf{Q}(5)^3 &i=12;\\
      \mathbf{Q}(4)^3 &i=11;\\
      \mathbf{Q}(3) &i=10;\\
      0 &\mbox{otherwise.}
   \end{cases} 
\end{equation*}

Similarly to \cite[Section 3.4]{Tom}, the top degree class of $\oline{H}_\bullet (F_{8c};\mathbf Q)$ is anti-invariant for $\nu.$

\subsubsection{The whole configuration space of type 8}

In oder to determine the Borel-Moore homology of $X_8$ from that of $X_{8x}$ and $X_{8c},$ we consider an auxiliary configuration space $Y$ which parametrizes irreducible conics on $Q$ through $p$ and three concurrent lines in $\mathbf{P}^3$ such that 
\begin{itemize}
    \item two of them intersects the conic at two pairs of points;
    \item the other one intersects $Q$ at two points and does not intersect the conic.
\end{itemize}
There is a proper forgetful map $f:Y\longrightarrow X_{8a}\cup X_{8c}$ which forgets the irreducible conics.

  The space $X_{8c}$ is embedded in both $Y$ and $X_{8a}\cup X_{8c}$. We have then the map between excision pairs $(Y,Y\setminus X_{8c})$ and $(X_{8a}\cup X_{8c}, X_{8a})$, i.e. the following commutative diagram.
  \begin{center}
      \begin{tikzcd}
          Y\setminus X_{8c}\arrow[r,"i'"]\arrow[d,"f"] &Y\arrow[d,"f"]\\
          X_{8a}\arrow[r,"i"] &X_{8a}\cup X_{8c}
      \end{tikzcd}
  \end{center}

Note that the map $f:Y\setminus X_{8c}\xrightarrow[]{}X_{8a}$ is a double cover branched over the closed loci of configurations containing $p$ as intersection of an irreducible conic with a reducible conic. Thus, it induces a surjective map on Borel-Moore homology.  Then the above diagram gives the following commutative diagrams on the excision sequences:
\begin{center}\begin{equation}\label{exc}
    \begin{tikzcd}
        ... \arrow[r]& \oline{H}_i(Y;\pm\mathbf{Q})\arrow[r]\arrow[d] &\oline{H}_i(X_{8c};\pm\mathbf{Q})\arrow[r,"\delta_i'"]\arrow[d,equal]&\oline{H}_{i-1}(Y\setminus X_{8c};\pm\mathbf{Q})\arrow[r]\arrow[d,two heads] &...\\
        ...\arrow[r] &  \oline{H}_i(X_{8a}\cup X_{8c};\pm\mathbf{Q})\arrow[r] & \oline{H}_i(X_{8c};\pm\mathbf{Q})\arrow[r,"\delta_i"]&\oline{H}_{i-1}(X_{8a};\pm\mathbf{Q})\arrow[r] &...\\
    \end{tikzcd}\end{equation}
\end{center}

We can now understand the boundary map $\delta_i$ by investigating the boundary map $\delta_i'$. We first notice that the embedding $X_{8c}\xhookrightarrow{} Y$ is compatible with the fiber bundle description in the subsection for type 8c. Indeed, $Y$ is a fiber bundle over $\check{p}\setminus(\check{\ell_1}\cup\check{\ell_2})$ and the fiber of $Y\xrightarrow[]{}\check{p}\setminus(\check{\ell_1}\cup\check{\ell_2})$ is fibered over $F(\mathbf{P}^1,4)/L$ with fiber isomorphic to $\mathbf{P}^2\setminus(C\cup\ell)$ where $\ell$ intersects $C$ at two points. We denote the fibers of $X_{8c}$ resp. $Y$ over $\check{p}\setminus(\check{\ell_1}\cup\check{\ell_2})$ by $\mathcal{F}_{8c}$ resp. $\mathcal{F}_Y$. By considering $\mathbf{C}$ as $\mathbf{P}^1\setminus\{p\}$, we have the following commutative diagram.
  \begin{center}
\begin{tikzcd}
  \mathcal{F}_{8c} \arrow[r,"i'"]\arrow[d] & \mathcal{F}_Y\arrow[d]\\
  M:=F(\mathbf{C},4)/L\arrow[r,"i "]& F(\mathbf{P}^1,4)/L
  \end{tikzcd} 
  \end{center}

  We denote the closed subspaces $(F(\mathbf{P}^1,4)\setminus F(\mathbf{C},4))/L\subset F(\mathbf{P}^1,4)/L$ resp. $\mathcal{F}_Y\setminus\mathcal{F}_{8c} \subset \mathcal{F}_Y$ by $\partial M$ resp. $\partial \mathcal{F}_{8c}$. By \cite[Lemma 3.3]{Tom}, the Borel-Moore homology of $F(\mathbf{P}^1,4)/L$ with twisted coefficient is $\mathbf{Q}(1)$ in degree $4$ and $\mathbf{Q}(3)$ in degree $7$. We obtain the short exact sequence from the excision sequence associated with $(F(\mathbf{P}^1,4)/L,\partial M)$: 
\begin{align*}
        \oline{H}_6(F(\mathbf{C},4)/L;\pm\mathbf{Q})\xrightarrow[]{\sim} \oline{H}_5(\partial M;\pm\mathbf{Q}).
\end{align*}
This implies that 
\begin{align*}   \oline{H}_8(\mathcal{F}_{8c};\pm\mathbf{Q})&\xrightarrow[]{\sim} \oline{H}_7(\partial \mathcal{F}_{8c};\pm\mathbf{Q}).
\end{align*}
By considering the fiber bundle structures of $X_{8c}$ and $Y$ over $\check{p}\setminus(\check{\ell_1}\cup\check{\ell_2})$, we have
\begin{align*}       
\delta'_{12}:\oline{H}_{12}(X_{8c};\pm\mathbf{Q})&\xrightarrow[]{\sim} \oline{H}_{11}(Y\setminus X_{8c};\pm\mathbf{Q}),
\end{align*} 
from the excision sequence associated with $(Y,Y\setminus X_{8b})$.
This implies that $\delta_{12}:\oline{H}_{12}(X_{8c};\pm\mathbf{Q})\xrightarrow[]{}\oline{H}_{11}(X_{8a};\pm\mathbf{Q})$ is surjective.

One can check that, because of the divisibility argument of Theorem \ref{thm:gen_leray_hirsch},
Because of the divisibility property 
As a result, we have that the boundary maps $\delta_i$ in \eqref{exc} are surjective for $i=11,10$ as well.
As a result, we have
\begin{equation*}
     \oline{H}_i(X_8;\pm\mathbf{Q})=\begin{cases}
         \mathbf{Q}(6) & i=13;\\
         \mathbf{Q}(5)^2 & i=12;\\
         \mathbf{Q}(4)  & i=11;\\
         0 & \mbox{ otherwise.}
     \end{cases}
      \end{equation*}

\subsection{Column (9)}
The space $X_9$ consists of configurations of seven points on the quadric surface $Q$, such that one point is the intersection of two lines of distinct rulings $\ell_1,\ell_2$ and the other six points are just the intersections among $\ell_1\cup\ell_2$ and two other conics. In addition, we require that the fixed point $p$ lies on the union of the rulings and conics. Notice that the first point in a configuration can be actually recovered from the other six points which can be considered as the intersections of three concurrent lines in $\mathbf{P}^3$ with $Q$, such that two of them lie on a tangent plane to $Q$. We can have the following possibilities:
\begin{itemize}
    \item[(9a)] The point $p$ lies on a line of the rulings.
    \item[(9b)] The point $p$ is in the configuration and it is an intersection point of the two conics.
    \item[(9c)] The point $p$ is {\em not} in the configuration and it lies on one of the conics.
\end{itemize}
\subsubsection{Configuration of type 9a}
The point $p$ uniquely determines (at least) a line $l$ of the ruling passing through it.
The configuration space $X_{9a}$ is then fibered over the space parametrizing 3-points configurations on $l$. The first point corresponds to the intersection with the line of the other ruling while the other two points are the intersections with the two conics. The conics are then determined by the intersection points with each other because a conic on $Q$ is uniquely determined by three points. Thus, the fiber will be isomorphic to $(F(\mathbf{C},2)\times F(\mathbf{C},2))/\mathfrak{S}_2$ because the intersection points of the two conics on $Q$ lies on the complement of the two rulings and they cannot both lie on any ruling. The coefficient $\pm\mathbf{Q}$ on $X_{9a}$ restricted to the fiber will be also denoted by $\pm\mathbf{Q}$. It is locally isomorphic to $\mathbf{Q}$ but changes sign with respect to the action of $\mathfrak{S}_2$. As $\oline{H}_\bullet(B(\mathbf{C},2);\pm\mathbf{Q})=0$, we also have $\oline{H}_\bullet((F(\mathbf{C},2)\times F(\mathbf{C},2))/\mathfrak{S}_2;\pm\mathbf{Q})=0$ and $\oline{H}_\bullet(X_{9a};\pm\mathbf{Q})=0$.

\subsubsection{Configuration of type 9b}
The configuration space of type $9b$ can be just interpreted similarly to that of type $8a$, i.e. it corresponds to three concurrent lines in $\mathbf{P}^3$ such that one of them through $p$. Furthermore, we require that the other two lines lie on a tangent plane to $Q$.
Thus, $X_{9b}$ is fibered over $\mathbf{P}^3\setminus (Q\cup T_pQ)$ which is the space parametrizing the choices of the vertex $q$. The space parametrizing the lines through $q$ is just $\mathbf{P}(T_q\mathbf{P}^3)\cong \mathbf{P}^2$ while the subspace containing lines through $q$ tangent to $Q$ is a conic $C\subset \mathbf{P}^2$. Once we pick the vertex $q$, we need to pick a tangent plane to $Q$ containing $q$ but not $pq$. This corresponding to picking a tangent line to $C\subset\mathbf{P}^2$ but not through $[pq]\notin C$. The space parametrizing such choices is just the conic $C$ without two points, i.e. isomorphic to $\mathbf{C}^*$. The last thing to choose is two points on the chosen tangent line but not on $C$. The space parametrizing such choice is just isomorphic to $B(\mathbf{C},2)$. Hence, the fiber of $X_{9b}\xrightarrow[]{} \mathbf{P}^3\setminus (Q\cup T_pQ)$ is isomorphic to $\mathbf{C}^*\times B(\mathbf{C},2)$. Notice that the coefficient $\pm\mathbf{Q}$ is constant on the base $\mathbf{P}^3\setminus (Q\cup T_pQ)$, and the fiber $\mathbf{C}^*\times B(\mathbf{C},2)$ because moving the vertex, and the tangent plane will not change the orientation of the simplicial bundle over $X_{9b}$ while interchanging the lines will permutate two pairs of points. By tensoring $\oline{H}_\bullet(\mathbf{P}^3\setminus (Q\cup T_pQ);\mathbf{Q})$, $\oline{H}_\bullet(\mathbf{C}^*;\mathbf{Q})$ and $\oline{H}_\bullet(B(\mathbf{C},2);\mathbf{Q})$, we have
\begin{equation*}
    \oline{H}_i(X_{9b};\pm\mathbf{Q})=\begin{cases}
    \mathbf{Q}(6) &i=12;\\
    \mathbf{Q}(5)^3 &i=11;\\
    \mathbf{Q}(4)^3 &i=10;\\
    \mathbf{Q}(3) &i=9;\\
    0 &\mbox{otherwise.}
    \end{cases}
\end{equation*}

Notice that all top degree classes of $\oline{H}_\bullet(\mathbf{P}^3\setminus (Q\cup T_pQ);\mathbf{Q})$, $\oline{H}_\bullet(\mathbf{C}^*;\mathbf{Q})$ and $\oline{H}_\bullet(B(\mathbf{C},2);\mathbf{Q})$ are invariant with respect to $\nu,$ while all the other classes are anti-invariant. This yields that $\oline{H}_{12}(X_{9b};\mathbf Q)$ is $\nu$-invariant while all classes in $\oline{H}_{11}(X_{9,b};\mathbf Q)$ are anti-invariant.
Finally, the action of $\nu$ interchanges two pairs of points. Hence, the Borel-Moore homology of the fiber of the bundle $\Phi_{9b}\rightarrow X_{9b}$ is also invariant for $\nu.$ We conclude that $\oline{H}_{20}(F_{9b};\mathbf Q)$ and $\oline{H}_{19}(F_{9b};\mathbf Q)$ are respectively invariant and anti-invariant, for $\nu.$

\subsubsection{Configuration of type 9c}
The space $X_{9c}$ is a fiber bundle over the space $$A=\{(v,q)\in Q\times \mathbf{P}^3: q \in H_v\setminus Q \mbox{ and } p\notin H_v \},$$ where $H_v$ is the tangent plane of $Q$ at $v$. The space $A$ parametrizes the choice of the hyperplane containing the two lines of distinct rulings on $Q$ and the intersection point of the three concurrent lines in $\mathbf{P}^3$. The fiber $\mathcal F$ of $X_{9c}$ over a point $(v,q)$ parametrizes three lines through $q$ such that two of them lie on $T_v Q$ and the marked point $p$ lies on the hyperplane spanned by the remaining line and one of the two lines on $T_vQ$. The fiber space is isomorphic to the configuration space of picking three points on $\mathbf{P}(T_q\mathbf{P}^3)\cong \mathbf{P}^2$ minus the conic $C$ such that the configuration consists of
\begin{itemize}
    \item a point $x_1$ does not lie on the tangent line $\ell_1$ to $C$ at $[qv]$ and not on the line $\ell_2$ through $[qv]$ and $[qp]$;
    \item a point $x_2$ which is the intersection of $\ell_1$ with the line through $[qx_1]$ and $[qp]$;
    \item a point $x_3$ on $\ell_1\setminus \{[qv],x_2\}$.
\end{itemize}
The requirement of $x_1$ not lying on $\ell_2$ is to rule out the case that the conic on $Q$ containing $p$ passes through the point $v$, which is the intersection point of the two lines of distinct rulings. The description above implies that $\mathcal F$ is again fibered over $\mathbf{P}^2\setminus(C\cup \ell_1\cup \ell_2)$ with fiber isomorphic to $\mathbf{C}^*$. Notice that the $\mathcal F$ is an oriented $\mathbf C^*$-bundle because the point $[qv]$ removed from $\ell_1$ is a fixed point. Thus, the local systems on $\mathbf{P}^2\setminus(C\cup \ell_1\cup \ell_2)$ whose stalks are $\oline{H}_2(\mathbf{C}^*;\mathbf{Q})$ and $\oline{H}_1(\mathbf{C}^*;\mathbf{Q})$ are just the constant coefficient $\mathbf{Q}$. This implies that if $\oline{H}_\bullet(\mathbf{P}^2\setminus (C\cup\ell_1\cup\ell_2);W)=0$ then also $\oline{H}_\bullet(X_{9c};\pm\mathbf Q)=0$. Furthermore, $W$ restricted to $\ell_2\setminus 2\mbox{pts}\cong \mathbf{C}^*$ is just $T$ which is locally isomorphic to $\mathbf{Q}$ but changes sign along a simple loop around the origin. Recall that $\oline{H}_\bullet(\mathbf{C}^*;T)=0$. We consider the following long exact sequence:
$$...\rightarrow \oline{H}_k(\ell_2\setminus 2\mbox{pts};W)\rightarrow \oline{H}_k(\mathbf{P}^2\setminus (C\cup\ell_1);W)\rightarrow \oline{H}_k(\mathbf{P}^2\setminus (C\cup\ell_1\cup\ell_2);W)\rightarrow ... $$
By Lemma~\ref{lm:inj_incl}, $\oline{H}_\bullet(\mathbf{P}^2\setminus (C\cup\ell_1);W)=0$ and thus $\oline{H}_\bullet(\mathbf{P}^2\setminus (C\cup\ell_1\cup\ell_2);W)=0.$ As a result, we can conclude that $\oline{H}_\bullet(X_{9c};\pm\mathbf Q)=0$ and $\oline{H}_\bullet(X_9;\pm\mathbf Q)\cong \oline{H}_\bullet(X_{9b};\pm\mathbf Q)$.

\subsection{Column (10) (F)}
Depending on the position of the point $p$ with respect to the singular curve, we the following cases:

\begin{itemize}
    \item[(10a)] the point $p$ is a point of the conic; 
    \item[(10b)] the point $p$ does not lie on the conic, which is degenerate; 
    \item[(10c)] the point $p$ does not lie on the conic, which is irreducible.
\end{itemize}

For each of these possibilities we compute the twisted Borel-Moore homology of the corresponding configuration space and then consider the spectral sequence associated to this stratification in order to obtain the twisted Borel-Moore homology of $X_{10}.$

\subsubsection{Configuration of type 10a}
The computation of the twisted Borel-Moore homology of the space $X_{10a}$ is very similar to that in \cite[Section 3.5]{Tom}. The space $X_{10a}$ can be considered as a fiber bundle over the space of conics passing through $p.$
If the conic were degenerate, we would have a fiber isomorphic to either $B(\mathbf C,2)\times B(\mathbf C^*,2)$ or $B(\mathbf C,2)\times B(\mathbf C,2).$ The twisted Borel-Moore homology of $B(\mathbf C,2)$ is trivial, \cite[Lemma 2]{Vas}, therefore we may assume the conic to be irreducible and the base space of the fiber bundle isomorphic to $T_p(\mathbf P^3)\setminus (\check{Q}\cap\check{p}),$ whose twisted Borel-Moore homology is $\mathbf Q(2)$ in degree $4,$ $\mathbf Q(1)$ in degree $3$ and zero in all other degrees.
The fiber and the local system of coefficients we need to consider are precisely the same as in \cite[3.5]{Tom}. Then by \cite[Lemmas 3.2 and 3.3]{Tom} we have that $\oline{H}(X_{10a};\pm\mathbf Q)$ is precisely that represented in the first column of Table \ref{column_10}. Notice also that also in this case the top degree class of $\oline{H}_\bullet(F_{10a})$ is anti-invariant with respect to $\nu$.

\subsubsection{Configuration of type 10b}
The space $X_{10b}$ can be considered as a fiber bundle over the space of degenerate conics not passing through $p.$

The base space is isomorphic to $\check Q\setminus (\check{Q}\cap\check{p})\cong \mathbf C^2.$
The fiber is then defined by pairs of lines of rulings, at least one passing though $p.$ The space parametrizing the pairs of lines of the ruling such that exactly one passes through $p$ is isomorphic to the configuration of $3$ points on each line on the degenerate conic, minus the singular point. In instead $p$ is the point of intersection of two lines of distinct rulings, then two lines are uniquely determined and each of the other two corresponds to a point on the component of the degenerate conic corresponding to the other ruling, minus the singular point, and the point corresponding to the line through $p.$ The fiber then is isomorphic to $(B(\mathbf C,2)\times B(\mathbf C,2))^{\oplus2}\setminus (\mathbf{C}^*\times \mathbf{C}^*).$ 

The twisted Borel-Moore homology of $X_{10b}$ then is just the tensor product of the twisted Borel-Moore homology of the base space and of the fiber, represented in the second column of Table \ref{column_10}.

\subsubsection{Configuration of type 10c}
The space $X_{10c}$ can be considered as a fiber bundle over the space of irreducible conics not passing through $p.$ We have two choices for a line through $p$ that determines a line in the configuration. The space of the remaining lines is then given by $F(\mathbf C,2)$ the configurations of three points on the conic, distinct from the intersection with the line through $p$. We are also double counting the case in which $p$ is the intersection of two lines. The fiber $\Psi$ is then isomorphic to the complement of $F(\mathbf{C}^*,2)$ in a space that is itself isomorphic to two disjoint copies of $F(\mathbf C,3),$ let us denote by $i$ this inclusion.

The local system $R$ that we need to consider on the fiber is the one changing its sign along the loops interchanging a pair of points corresponding to two lines of the same ruling, not through $p.$ Therefore we have that $\oline{H}_k(\Psi; R)$ can be computed from the exact sequence 
$$\dots\rightarrow\oline{H}_k(F(\mathbf{C}^*,2);\mathbf Q)\rightarrow \oplus_{i=1}^2\oline{H}_k(F(\mathbf C,3);i_*R)\rightarrow \oline{H}_k(\Psi; R)\rightarrow\dots$$
where $\oline{H}_\bullet(F(\mathbf C,3);i_*R)$ is the invariant part of $\oline{H}_\bullet(F(\mathbf C,3);\mathbf Q)$ with respect to the action of a transposition $(12).$ By \cite[Lemma 2.12]{Tom} and Lemma \ref{lm:FC_34}, we have that $\oline{H}_k(\Psi; R)$is $\mathbf{Q}(2)^3$ in degree $5$, $\mathbf{Q}(1)^5$ in degree $4$, $\mathbf{Q}^2$ in degree $3$ and zero in all other degrees. Tanking the tensor product with the Borel-Moore homology of the base space yields the description in the third column of Table \ref{column_10}.
Notice also that, from a similar argument to \cite[Column]{Tom}, the top degree class of $F_{10c}$ is also anti-invariant for $\nu$.

\begin{table}[ht!]\caption{Spectral sequence converging to $\oline{H}_\bullet(X_{10};\pm\mathbf Q)$}\centering
\begin{tabular}{c|ccc}
$10$&$\mathbf Q(5)$&&\\
$9$&$\mathbf Q(4)$&&\\
$8$&&&$\mathbf Q(5)^3$\\
$7$&$\mathbf Q(3)$&$\mathbf Q(4)$&$\mathbf Q(4)^8$\\
$6$&$\mathbf Q(2)$&$\mathbf Q(3)^2$&$\mathbf Q(3)^7$\\
$5$&&$\mathbf Q(2)$&$\mathbf Q(2)^2$\\
\hline
&(10a)&(10b)&(10c)
\end{tabular}
\label{tb: column_10}
\end{table}

Let us notice that the differentials in Table \ref{tb: column_10} have maximal rank because of the divisibility argument of Theorem \ref{thm:gen_leray_hirsch}. 
Then the twisted Borel-Moore homology of $X_{10}$ is $\mathbf{Q}(5)^4$ in degree $11$, $\mathbf{Q}(4)^8$ in degree $10$, $\mathbf{Q}(3)^4$ in degree $9$ and zero otherwise.

The Borel-Moore homology of $F_{10}$ is $\mathbf{Q}(6)^4$ in degree $20$, $\mathbf{Q}(5)^8$ in degree $19$, $\mathbf{Q}(4)^4$ in degree $18$ and zero otherwise. Also recall that since both top degree classes of $F_{10a}$ and $F_{10c}$ are anti-invariant with respect to $\nu,$ the classes $\oline{H}_{20}(F_{10};\mathbf Q)$ are also $\nu$-anti-invariant.

\subsection{Column 11}

By Proposition \ref{prop:Gor} the space $F_{11}$ is an open cone and the Borel-Moore homology of its base space can be computed from the spectral sequence in Table \ref{opencone_11}.

\begin{table}[ht!]\caption{Spectral sequence converging to the Borel-Moore homology of the base of $F_{11}.$}\centering
\begin{tabular}{c|ccccccccccc}
$12$&&&&&&&&$\mathbf Q(7)$&&\\
$11$&&&&&&&&$\mathbf Q(6)^2$&$\mathbf Q(7)$&\\
$10$&&&&&&&&$\mathbf Q(5)$&$\mathbf Q(6)^3$&$\mathbf Q(6)^4$\\
$9$&&&&&&&&&$\mathbf Q(5)^3$&$\mathbf Q(5)^8$\\
$8$&&&&&&&&&$\mathbf Q(4)$&$\mathbf Q(4)^4$\\
$7$&&&&&&&&&&\\
$6$&&&&&&&&&&\\
$5$&&&&&&&&&&\\
$4$&&&&&&$\mathbf{Q}(4)$&$\mathbf{Q}(4)$&&&\\
$3$&&&&$\mathbf{Q}(3)^2$&$\mathbf{Q}(3)^2$&&&&&\\
$2$&&$\mathbf{Q}(2)$&$\mathbf{Q}(2)$&&&&&&&\\
$1$&&&&$\mathbf{Q}(2)$&$\mathbf{Q}(2)$&&&&\\
$0$&&$\mathbf{Q}(1)^2$&$\mathbf{Q}(1)^2$&&&&&&&\\
$-1$&$\mathbf{Q}$&&&&&&&&&\\
\hline
&$(1)$&$(2)$&$(3)$&$(4)$&$(5)$&$(6)$&$(7)$&$(8)$&$(9)$&$(10)$
\end{tabular}
\label{opencone_11}
\end{table}
The differentials among the first $7$ columns are all of maximal rank for dimensional reasons. The space $X_0$ is affine of dimension $15$ and if any of these differential were not trivial, it would give a non-trivial element in $\oline{H}_{j}(\Sigma_0;\mathbf Q)$, with $j\leq 11,$ 
which corresponds, by Alexander duality, to a non-trivial element in ${H}^{k}(\Sigma_0;\mathbf Q)$, with $k= 30-j-1>15.$ 

The differentials among the last three columns, instead, are all trivial because otherwise this would contradict the divisibility by the cohomology of the group. 

The Borel-Moore homology of $F_{11}$ then is 
$\mathbf Q(6)^2+\mathbf Q(5)^4$ in degree $19,$ $\mathbf Q(5)^5+\mathbf Q(4)^8$ in degree $18,$ $\mathbf Q(4)^4+\mathbf Q(3)^4$ in degree $17,$ $\mathbf Q(3)$ in degree $16,$ and zero in all other degrees.

\subsection{Vassiliev's spectral sequence}
Here in Table \ref{tb:main} is represented the spectral sequence associated to the projection $\pi_0:\mathcal I_0\rightarrow Z_0\cong\mathbf P^1\times \mathbf P^1,$ where the cohomology of the fiber if the one obtained in the previous subsections.

\begin{table}[ht!]\caption{$E_2$-term of the spectral sequence converging to ${H}^\bullet(\mathcal I_0;\mathbf Q).$}\centering\footnotesize
\begin{tabular}{c|ccccc}
$13$&$\mathbf Q(-12)$&&$\mathbf Q(-13)^2$&&$\mathbf Q(-14)$\\
$12$&$\mathbf Q(-11)^{5}+\mathbf Q(-12)^4$&&$\mathbf Q(-12)^{10}+\mathbf Q(-13)^8$&&$\mathbf Q(-13)^{5}+\mathbf Q(-14)^4$\\
$11$&$\mathbf Q(-10)^9+\mathbf Q(-11)^{12}$&&$\mathbf Q(-11)^{18}+\mathbf Q(-12)^{24}$&&$\mathbf Q(-12)^9+\mathbf Q(-13)^{12}$\\
$10$&{$\mathbf Q(-7)$}$+$ {$\mathbf Q(-9)$} $^7+\mathbf Q(-10)^{12}$&&$\mathbf Q(-8)^2+\mathbf Q(-10)^{14}+\mathbf Q(-11)^{24}$&&$\mathbf Q(-9)+\mathbf Q(-11)^7+\mathbf Q(-12)^{12}$\\
$9$&{$\mathbf Q(-6)$}$+${$\mathbf Q(-8)^2+\mathbf Q(-9)^4$}&&$\mathbf Q(-7)^2+\mathbf Q(-9)^4+\mathbf Q(-10)^8$&&$\mathbf Q(-8)+\mathbf Q(-10)^2+\mathbf Q(-11)^4$\\
$8$&&&&&\\
$7$&{$\mathbf Q(-5)^3$}&&$\mathbf Q(-6)^6$&&{$\mathbf Q(-7)^3$}\\
$6$&{$\mathbf Q(-4)^3$}&&$\mathbf Q(-5)^6$&&{$\mathbf Q(-6)^3$}\\
$5$&&&&&\\
$4$&{$\mathbf Q(-3)^3$}&&$\mathbf Q(-4)^6$&&{$\mathbf Q(-5)^3$}\\
$3$&{$\mathbf Q(-2)^3$}&&$\mathbf Q(-3)^6$&&{$\mathbf Q(-4)^3$}\\
$2$&&&&&\\
$1$&$\mathbf Q(-1)$&&$\mathbf Q(-2)^2$&&$\mathbf Q(-3)$\\
$0$&$\mathbf Q$&&$\mathbf Q(-1)^2$&&$\mathbf Q(-2)$\\
\hline
&$0$&$1$&$2$&$3$&$4$
\end{tabular}
\label{tb:main}
\end{table}

The rank of the differentials in the rows 9--13 is determined by the divisibility argument of Theorem \ref{thm:gen_leray_hirsch}.

This yields that the fourth page of the spectral sequence is the one represented in Table \ref{tb:main2}.

\begin{table}[ht!]\caption{$E_4$-term of the spectral sequence converging to ${H}^\bullet(\mathcal I_0;\mathbf Q).$}\centering\footnotesize
\begin{tabular}{c|ccccc}
$13$&&&&&$\mathbf Q(-14)$\\
$12$&&&$\mathbf Q(-12)^{2}$&&$\mathbf Q(-13)^{3}+\mathbf Q(-14)^4$\\
$11$&$\mathbf Q(-10)$&&$\mathbf Q(-11)^{6}+\mathbf Q(-12)^{8}$&&$\mathbf Q(-12)^2+\mathbf Q(-13)^{4}$\\
$10$&$\mathbf Q(-7)+$\colorbox{yellow!25}{$\mathbf Q(-9)^3$} $+\mathbf Q(-10)^{4}$\tikzmark{x1}&&$\mathbf Q(-8)^2+\mathbf Q(-10)^{4}+\mathbf Q(-11)^{8}$&&$\mathbf Q(-9)$\\
$9$&$\mathbf Q(-6)+$\colorbox{yellow!25}{$\mathbf Q(-8)^2+\mathbf Q(-9)^4$}\tikzmark{x2}&&$\mathbf Q(-7)^2$&&$\mathbf Q(-8)$\\
$8$&&&&&\\
$7$&$\mathbf Q(-5)^3$\tikzmark{x3}&&$\mathbf Q(-6)^6$&&\tikzmark{y1}{$\mathbf Q(-7)^3$}\\
$6$&$\mathbf Q(-4)^3$\tikzmark{x4}&&$\mathbf Q(-5)^6$&&\tikzmark{y2}{$\mathbf Q(-6)^3$}\\
$5$&&&&&\\
$4$&{$\mathbf Q(-3)^3$}\tikzmark{x5}&&$\mathbf Q(-4)^6$&&\tikzmark{y3}{$\mathbf Q(-5)^3$}\\
$3$&{$\mathbf Q(-2)^3$}\tikzmark{x6}&&$\mathbf Q(-3)^6$&&\tikzmark{y4}{$\mathbf Q(-4)^3$}\\
$2$&&&&&\\
$1$&$\mathbf Q(-1)$&&$\mathbf Q(-2)^2$&&\tikzmark{y5}$\mathbf Q(-3)$\\
$0$&$\mathbf Q$&&$\mathbf Q(-1)^2$&&\tikzmark{y6}{$\mathbf Q(-2)$}\\
\hline
&$0$&$1$&$2$&$3$&$4$
\end{tabular}
\begin{tikzpicture}[overlay, remember picture, yshift=.25\baselineskip, shorten >=.5pt, shorten <=.5pt]
	\draw [shorten >=.1cm,shorten <=.1cm,->]([yshift=-5pt]{pic cs:x1}) -- ([yshift=4pt]{pic cs:y1});
	\draw [shorten >=.1cm,shorten <=.1cm,->]([yshift=-1pt]{pic cs:x2}) --  ([yshift=3pt]{pic cs:y2});
	\draw [shorten >=.1cm,shorten <=.1cm,->]([yshift=2.5pt]{pic cs:x3}) -- ([yshift=2.5pt]{pic cs:y3});
    \draw [shorten >=.1cm,shorten <=.1cm,->]([yshift=2.5pt]{pic cs:x4}) -- ([yshift=2.5pt]{pic cs:y4});
	\draw [shorten >=.1cm,shorten <=.1cm,->]([yshift=2.5pt]{pic cs:x5}) --  ([yshift=2.5pt]{pic cs:y5});
	\draw [shorten >=.1cm,shorten <=.1cm,->]([yshift=2.5pt]{pic cs:x6}) -- ([yshift=2.5pt]{pic cs:y6});
    \end{tikzpicture}
\label{tb:main2}
\end{table}

Moreover the differentials between column $0$ and $4$ are all trivial except for the ones highlighted in Table \ref{tb:main2}. 
We will now prove that 
\begin{equation}\label{diff}
\operatorname{rk}E_4^{0,3+i}\to E_4^{4,i}=\begin{cases}1,& i=0,1,6,7; \\
2,& i=3,4;\\
0,&\text{ otherwise.}
\end{cases}
\end{equation}

\begin{lm}
    The rank of differential $d^4_{0,3}:E^4_{0,3}\to E^4_{4,0}$ in Table \ref{tb:main2} is $1$.
\end{lm}

\begin{proof}
Let us consider $$\mathcal{V}=\{(f,p)\in V\times (\mathbf{P}^1\times\mathbf{P}^1): f(p)=0\}$$
$$W=\{(f,p)\in V\times (\mathbf P^1\times\mathbf P^1): f \text{ singular at } p\},$$
with $V=\Gamma(\mathcal{O}_{\mathbf{P}^1}(3)\otimes \mathcal{O}_{\mathbf{P}^1}(3)).$
Both spaces have a natural projection to $\mathbf P^1\times\mathbf P^1,$ with fibers $\mathcal V_p\cong\mathbf{C}^{15}$ and $W_p\cong\mathbf C^{13},$ respectively.

Their Borel-Moore homology can be easily compute and the inclusion $W\xhookrightarrow{i}\mathcal V$ yields the  long exact sequence
$$\dots\to\oline{H}_{k}(W)\to\oline{H}_{k}(\mathcal V)\to\oline{H}_{k}(\mathcal V\setminus W)\to\dots$$
We claim that the induced map 
$
i_*:\oline{H}_{30}(W)\to\oline{H}_{30}(\mathcal V)
$
has rank $1.$ 
 
Let us write $W,$ as 
$$W=\{(f,p)\in V\times (\mathbf P^1\times\mathbf P^1): \frac{\partial f}{\partial x_0}(p)=\frac{\partial f}{\partial x_1}(p)=\frac{\partial f}{\partial y_0}(p)=\frac{\partial f}{\partial y_1}(p)=0,\},$$ which has codimension $3$ in $V\times\mathbf{P}^1\times\mathbf{P}^1$, by the Euler Formula.
Let $M_1,M_2\subset\mathcal V$ be the codimension one subvarieties defined by $\frac{\partial f}{\partial x_0}(p) =\frac{\partial f}{\partial x_1}(p)=0$ and $\frac{\partial f}{\partial y_0}(p)= \frac{\partial f}{\partial y_1}(p)=0$, respectively. Since any two conditions from distinct pairs are independent, $M_1$ intersects with $M_2$ transversally and $W=M_1\cap M_2$ which yields that $i_*[W]$ is just the cup product $[M_1]\smile [M_2]$. 
Let $[X],[Y]\in H^2(V\times\mathbf{P}^1\times\mathbf{P}^1;\mathbf{Q})$ be the pullback of class of $[1,0]$ in each of the two copies of $\mathbf{P}^1$ via the projections:  $[X]\in H^0(V)\otimes H^2(\mathbf P^1)\otimes H^0(\mathbf P^1)$ and $[Y]\in H^0(V)\otimes H^0(\mathbf P^1)\otimes H^2(\mathbf P^1).$
Consider the hypersurfaces 
$$S_1=\{(f,p)\in V\times\mathbf{P}^1\times\mathbf{P}^1: \frac{\partial f}{\partial x_0}(p)=0\},\qquad S_2=\{(f,p)\in V\times\mathbf{P}^1\times\mathbf{P}^1: \frac{\partial f}{\partial y_0}(p)=0\}.$$
One can check that $S_1$ is rationally equivalent to the zero locus of $x_0^2y_0^3=0$, therefore
$[S_1]=2[X]+3[Y],$ and similarly $[S_2]=3[X]+2[Y].$
Moreover, by the Euler formula, the conditions defining $S_1,$ respectively $S_2,$ will cut out two reduced subspaces on $\mathcal{V}$ which are $M_1\cup \{[x_0,x_1]=[1,0]\},$ respectively $M_2\cup \{[y_0,y_1]=[1,0]\}$, i.e. $j^*[S_1]=[M_1]+j^*[X]$ and $j^*[S_2]=[M_2]+j^*[Y].$
Here $j$ denotes the inclusion of $\mathcal V$ in $V\times\mathbf{P}^1\times\mathbf{P}^1.$ 
Hence, $[M_1]=j^*[X]+3j^*[Y]$ and $[M_2]=3j^*[X]+j^*[Y]$. We have then $i_*[W]=[M_1]\smile[M_2]=10j^*([X]\smile[Y])$. The latter term is just some multiple of the fundamental class $[V_p]$ of the fiber of $\mathcal{V}\to \mathbf{P}^1\times \mathbf{P}^1$, which is non-trivial, from which injectivity of $i_*$ follows. 

This proves that $\oline H_{31}(\mathcal V\setminus W)=\oline H_{30}(\mathcal V\setminus W)=0$ and, by Poincar\'e duality, that the Leray spectral sequence of $\pi':\mathcal V\setminus W\rightarrow \mathbf P^1\times\mathbf P^1$ is the one represented in Table \ref{tb:rk1}, with the differential $E_4^{0,3}\to E_4^{4,0}$ of rank 1.
\begin{table}[ht!]\caption{Spectral sequence converging to ${H}^\bullet(\mathcal V\setminus W;\mathbf Q)$}\centering
\begin{tabular}{c|cccccc}
$3$&$\mathbf Q(-2)$&&$\mathbf Q(-3)^2$&&$\mathbf Q(-4)$\\
$2$&&&&&\\
$1$&&&&&\\
$0$&$\mathbf Q$&&$\mathbf Q(-1)^2$&&$\mathbf Q(-2)$\\
\hline
&0&1&2&3&4
\end{tabular}
\label{tb:rk1}
\end{table}

Finally, we consider the open inclusion $i_0:\mathcal I_0\to\mathcal V\setminus W$ and following the idea in \cite[Lemma 2.1]{Tom07}, we have a commutative diagram

\begin{center}
\begin{tikzcd}
\mathcal{I}_0\arrow[d,"\pi_0"]\arrow[r,"i_0"]&\mathcal V\setminus  W\arrow[d,"\pi'"]\\    \mathbf{P}^1\times\mathbf{P}^1\arrow[r,equal]&\mathbf{P}^1\times\mathbf{P}^1
\end{tikzcd}
\end{center}
which makes the differentials on the $E_4$ page of the spectral sequences of $\pi_0,\pi'$ commute.
In particular this proves $\operatorname{rk}d^4_{0,3}(\pi_0)=\operatorname{rk}d^4_{0,3}(\pi')=1$.
\end{proof}

The above lemma, together with the divisibility property, proves that all the other differential have rank as in \eqref{diff}.

The rational cohomology of $\mathcal I_0/G_0'$ is $\mathbf Q$ in degree $0$, $\mathbf Q(-1)^2$ in degree $2$, ($\mathbf Q(-2)$ in degree $3$, $\mathbf Q(-2)$ in degree $4$,) $\mathbf Q(-3)^2$ in degree $5$, $\mathbf Q(-4)$ in degree $7$, $\mathbf Q(-8)^2+\mathbf Q(-9)^4$ in degree $9$, $\mathbf Q(-9)$ in degree $10$,
and zero in all other degrees.

To conclude the proof This proves Theorem \ref{thm_C_0}, let us recall that the rational cohomology of $C_{0,1}$ is the $\nu$-invariant part of $H^\bullet(\mathcal I_0/G_0';\mathbf Q).$
The cohomology of $\mathbf{P}^1\times \mathbf{P}^1$ has only one anti-invariant class in $H^2(\mathbf{P}^1\times \mathbf{P}^1;\mathbf Q)$.
Instead, the classes $\mathbf Q(-8)^2+\mathbf Q(-9)^4$ in degree $9$, $\mathbf Q(-9)$ in degree $10$, highlighted  in Table \ref{tb:main2}, are defined by the strata $F_8,F_9$ and $F_{10}.$
From our previous discussion, among these classes, only the top degree class of $F_9$ is $\nu$-invariant, which is isomorphic to $\mathbf Q(-8),$ after taking the cap product with the fundamental class of the discriminant.

The rational cohomology of $C_{0,1}$ is then 
$\mathbf Q$ in degree $0$, $\mathbf Q(-1)$ in degree $2$, $\mathbf Q(-2)$ in degree $3$, $\mathbf Q(-2)$ in degree $4$, $\mathbf Q(-3)$ in degree $5$, $\mathbf Q(-4)$ in degree $7$, $\mathbf Q(-8)$ in degree $9$,
and zero in all other degrees.

\section{Curves on a quadric cone}
Following \cite[Section 4]{Tom}, we identify the quadric cone with the weighted projective plane $\mathbf P(1,1,2).$ Then consider the vector space $V_1\subset \mathbf C\left[x,y,z\right]_6;$ with $\operatorname{deg}x=\operatorname{deg}y=1$ and $\operatorname{deg}z=2,$ of polynomials whose vanishing loci are inside $\mathbf P(1,1,2)\setminus\{\left[0,0,1\right]\}$. We have that $\operatorname{dim}V_1=15$ and we define $X_1^p$ to be the open space in $V_1$ of polynomials defining non-singular curves and $\Sigma_1$ its discriminant. The automorphism group $G_1=\Aut(\mathbf C\left[x,y,z\right]_6)$ is homotopy equivalent to $G_1'=\GL(2)\times \mathbf{C}^*$ (c.f.\cite[Section 4.1]{Tom}), thus we will consider the cohomology of the geometric quotient $[\mathcal{I}_1/G_1']$.

Let $t=\left[0,0,1\right]\in\mathbf{P}(1,1,2)$ be the vertex of the cone and denote by $\ell_p$ the line containing $p.$ We say that a point is \emph{general} if it is distinct from $p$ and the vertex $t$. Two points are in general position if each of them is general and they are not collinear. 

\begin{table}[h]\caption{List of singular configurations}
\begin{tabular}{c|l|c}
Type&Description&Dimension\\
\hline
1& The vertex $t$.&$\left[14\right]$\\
2&The point $p. $&$\left[13\right]$\\
3&A general point. &$\left[12\right]$\\
4&The vertex $t$ and $p$ .& $\left[12\right]$\\
5&The vertex $t$ and a general point. &$\left[11\right]$\\
6&The point $p$ and a general point not on $tp$.&$\left[10\right]$\\
7&Two general points.&$\left[9\right]$\\
8&The vertex $t$, the point $p$ and a general point not on $tp$.&$\left[9\right]$\\
9&The vertex $t$ and two general points.&$\left[8\right]$\\

10a&The intersection of components of the union of&\\
&two lines and two conics, excluding the vertex, with $p$ on the curve. &$\left[1\right]$\\
10b&The intersection of components of the union of&\\
&two lines and two conics, with $p$ on the curve. &$\left[1\right]$\\
11&Six points: intersection of the cone with three concurrent lines, &\\
& with $p$ a point on the curve (which is the union of three conics).&$\left[1\right]$\\
12&The whole $\mathbf{P}(1,1,2)$.&$\left[0\right]$\\
\end{tabular}
\end{table}

\subsection{Columns (1)--(9)}
Again, from Proposition \ref{prop:Gor}, we can easily compute the first nine columns of the spectral sequence converging to the Borel-Moore homology of the discriminant.

The space $F_1$ is $\mathbf C^{14}$.

The space $F_2$ is $\mathbf C^{13}$.

The space $F_3$ is a $\mathbf C^{12}$-bundle over $S=\mathbf P(1,1,2)\setminus p$, which is quasi-isomorphic to $\mathbf{C}^2\sqcup \mathbf{C}^*$.

The space $F_4$ is $\mathbf C^{12}\times \mathring{\Delta}_1$.

The space $F_5$ is a $\mathbf C^{11}\times \mathring{\Delta}_1$-bundle over $S$.

The space $F_6$ is a $\mathbf C^{10}\times \mathring{\Delta}_1$-bundle over $\mathbf{P}(1,1,2)\setminus \ell_p\cong\mathbf{C}\times\mathbf{C}$.

The space $F_7$ is a $\mathbf{C}^9$-bundle over the space $\Phi_7$. The space $\Phi_7$ is a non-orientable $\mathring{\Delta}_1$-bundle over the space $X_7\subset B(S,2)$ consisting of configurations of two points not on the same line. The space $X_7$ is quasi-isomorphic to the disjoint union of $\mathbf{C}^2\times B(\mathbf{C},2)$ and $\mathbf{C}^2\times\mathbf{C}^*$. By Lemma \ref{lm_vas}, the twisted Borel-Moore of the first space is trivial. This implies that $\oline{H}_\bullet(X_7;\pm\mathbf{Q})$ is $\mathbf{Q}(3)$ in degree $6$, $\mathbf{Q}(2)$ in degree $5$. And $\oline{H}_\bullet(\Phi_7;\mathbf{Q})$ is $\mathbf{Q}(3)$ in degree $7$, $\mathbf{Q}(2)$ in degree $6$.

The space $F_8$ is a $\mathbf C^{9}\times \mathring{\Delta}_2$-bundle over $\mathbf{C}\times\mathbf{C}$.

The space $F_9$ is a $\mathbf{C}^8$-bundle over the space $\Phi_9$. The space $\Phi_9$ is a non-orientable $\mathring{\Delta}_2$-bundle over the space $X_9\subset B(S,2)$ consisting of configurations of two points not on the same line. Similar to Column (7), we have $\oline{H}_\bullet(X_9;\pm\mathbf{Q})=\oline{H}_\bullet(X_7;\pm\mathbf{Q})$. And $\oline{H}_\bullet(\Phi_9;\mathbf{Q})$ is $\mathbf{Q}(3)$ in degree $8$, $\mathbf{Q}(2)$ in degree $7$.

\subsection{Column (10a+10b)}
The space $F_{10a}\cup F_{10b}$ is a $\mathbf{C}$-bundle over $\Phi_{10a}\cup\Phi_{10b}$. In addition, $\Phi_{9a}\cup\Phi_{9b}$ is a fiber bundle over $X_{10a}=X_{10b}$ such that the fiber is just $\Delta_6$ minus all the facets adjacent to the vertex of $\Delta_6$ corresponding to $t$. The simplices $\Delta_6$ and the removed facets can be contracted to the vertex corresponding to $t$ simultaneously. Hence, by the same argument in \cite{Tom}, $\oline{H}_\bullet(\Phi_{10a}\cup\Phi_{10b};\mathbf{Q})=\oline{H}_\bullet(F_{10a}\cup F_{10b};\mathbf{Q})=0$.
\subsection{Column (11)}
The description here will be similar to that of Column (8) of $\Sigma_0$. First, we refine the configuration space $X_{11}$ as the disjoint union of two spaces, with the following description:
\begin{itemize}
    \item[(11a)] The point $p$ is in the configuration.
    \item[(11b)] The point $p$ is not in the configuration.
    
\end{itemize}
The space $X_{11a}$ is fibered over $\mathbf{P}^3$ minus the cone and the tangent plane to the cone at $p$. The base space parametrizes the choice of the meeting point $q$ of the concurrent lines in $\mathbf P^3$ and it is quasi-isomorphic to $\mathbf{C}^3\setminus \mathbf{C}^2$. The space of lines in $\mathbf{P}^3$ through $q$ not tangent to the cone is isomorphic to $\mathbf{P}^2\setminus (\ell_1\cup\ell_2)$, where $\ell_1,\ell_2$ are two lines on $\mathbf{P}^2$. The system of coefficient $T'$ we consider will be the one locally isomorphic to $\mathbf{Q}$ but alternating sign along a loop winding around $\ell_1$ or $\ell_2$. Since $\mathbf{P}^2\setminus (\ell_1\cup\ell_2)$ is isomorphic to $\mathbf{C}\times\mathbf{C}^*$ and $T'$ restricted to $\mathbf{C}^*$ will be just $T$, we have $\oline{H}_\bullet(\mathbf{P}^2\setminus (\ell_1\cup\ell_2);T')=0$. Notice that the fiber of the fibration of $X_{11a}$ is just $$\bigg((\mathbf{P}^2\setminus (\ell_1\cup\ell_2))^2\setminus \Delta\bigg)/\mathfrak{S}_2, $$
where $\Delta$ parametrizes the choices of $x_1,x_2$ from $\mathbf{P}^2\setminus (\ell_1\cup\ell_2)$ such that the $x_1,x_2,[pq]$ is collinear or any two of them define a line through $\ell_1\cap\ell_2$. We consider the filtration $\Delta_1\subset\Delta_2=\Delta$, where
\begin{align*}
    \Delta_1& =\{(x_1,x_2)\in (\mathbf{P}^2\setminus (\ell_1\cup\ell_2))^2: x_1,x_2 \mbox{ on the line through }[pq],\ell_1\cap\ell_2\}.
\end{align*}
The strata $\Delta_1$ and $\Delta_2\setminus\Delta_1$ contain a factor of $\mathbf{C}^*$ so they have trivial Borel-Moore homology with respect to $T'$. Hence, $\oline{H}_\bullet(X_{11a};\pm\mathbf{Q})=0.$

 The space $X_{11b}$ is fibered over $\check{p}\setminus \check{t}$, which parametrizes non-singular conics through $p$ on the cone. The fiber over a conic $C$ is again a fiber space over the space parametrizing configurations of four points on $C\setminus\{p\}$ quotient by the action of $\mathfrak{S}_2\times\mathfrak{S}_2$. The configuration of four points determines two of the concurrent lines so the fiber over it will be isomorphic to $\mathbf{P}^2\setminus(\ell_1\cup\ell_2\cup\ell_3)$ (here the three lines intersect at three points). Notice that $\mathbf{P}^2\setminus(\ell_1\cup\ell_2\cup\ell_3)$ is isomorphic to $\mathbf{C}^*\times\mathbf{C}^*$ so the Borel-Moore homology with respect to coefficient $T'$ restricted to $\mathbf{C}^*\times\mathbf{C}^*$ is trivial. As a result, $\oline{H}_\bullet(X_{11b};\pm\mathbf{Q})=0.$ And we can conclude that $\oline{H}_\bullet(\Phi_{11};\mathbf{Q})=\oline{H}_\bullet(F_{11};\mathbf{Q})=0$.

\subsection{Column (12)}
The space $F_{12}$ is an open cone over $\Phi_1\cup\dots\cup\Phi_{11}$. 
By dimension reason ($V_1\setminus\Sigma_1$ is of complex dimension $15$), all differentials in the associated spectral sequence would be trivial, therefore $\oline{H}_\bullet(F_{12};\mathbf{Q})=0.$

\subsection{The Vassiliev's spectral sequence}

The spectral sequence converging to the Borel Moore of the discriminant is represented in Table \ref{tb:ss_Sigma1}. 
\begin{table}[ht!]\caption{Spectral sequence converging to $\oline{H}_\bullet(\Sigma_1;\mathbf Q).$}\centering

\begin{tabular}{c|ccccccccc}
$27$&$\mathbf{Q}(14)$&&&&&&&&\\
$26$&&&&&&&\\
$25$&&&$\mathbf{Q}(14)$&&&&&&\\
$24$&&$\mathbf{Q}(13)$&&&&&&&\\
$23$&&&$\mathbf{Q}(13)$&&&&&&\\
$22$&&&$\mathbf Q(12)$&&$\mathbf{Q}(13)$&&&&\\
$21$&&&&$\mathbf{Q}(12)$&&&&&\\
$20$&&&&&$\mathbf{Q}(12)$&&&&\\
$19$&&&&&$\mathbf{Q}(11)$&$\mathbf{Q}(12)$&&&\\
$18$&&&&&&&$\mathbf{Q}(12)$&&\\
$17$&&&&&&&$\mathbf{Q}(11)$&&\\
$16$&&&&&&&&$\mathbf{Q}(11)$&\\
$15$&&&&&&&&&$\mathbf{Q}(11)$\\
$14$&&&&&&&&&$\mathbf{Q}(10)$\\

\hline
&$(1)$&$(2)$&$(3)$&$(4)$&$(5)$&$(6)$&$(7)$&$(8)$&$(9)$
\end{tabular}
\label{tb:ss_Sigma1}
\end{table}

Notice that $\mathbf{P}(1,1,2)\setminus \{t\}$ is quasi-isomorphic to $\mathbf{P}^1\times\mathbf{C}$ and from \eqref{div}, the Leray spectral sequence associated to $\pi_1:\mathcal{I}_1\to \mathbf{P}(1,1,2)\setminus \{t\}$ is represented in Table \ref{tb:leray_I_1} with all trivial differentials, except for those highlighted in the table.

\begin{table}[ht!]\caption{Spectral sequence converging to ${H}^\bullet(\mathcal I_1;\mathbf Q).$}\label{tb:leray_I_1}\centering
\begin{tabular}{c|ccccc}

$6$&$\mathbf Q(-5)$\tikzmark{a1}&&&&$\mathbf Q(-6)$\\
$5$&$\mathbf Q(-4)^4$\tikzmark{b1}&&&&\tikzmark{a2}$\mathbf Q(-5)^4$\\
$4$&$\mathbf Q(-3)^5$\tikzmark{c1}&&&&\tikzmark{b2}$\mathbf Q(-4)^5$\\
$3$&$\mathbf Q(-2)^2$&&&&\tikzmark{c2}$\mathbf Q(-3)^2$\\
$2$&$\mathbf Q(-2)$&&&&$\mathbf Q(-3)$\\
$1$&$\mathbf Q(-1)^2$&&&&$\mathbf Q(-2)^2$\\
$0$&$\mathbf Q(-0)$&&&&$\mathbf Q(-1)$\\ 
\hline
&$0$&&$1$&&$2$

\end{tabular}
    \begin{tikzpicture}[overlay, remember picture, yshift=.25\baselineskip, shorten >=.5pt, shorten <=.5pt]
	\draw [shorten >=.1cm,shorten <=.1cm,->]([yshift=2.5pt]{pic cs:a1}) -- ([yshift=2.5pt]{pic cs:a2});
    \node[text width=1cm]({pic cs:a2}) at (-2,0.6){1};
	\draw [shorten >=.1cm,shorten <=.1cm,->]([yshift=2.5pt]{pic cs:b1}) --  ([yshift=2.5pt]{pic cs:b2});
    \node[text width=1cm]({pic cs:b2}) at (-2,1.05){2};
    \node[text width=1cm]({pic cs:c2}) at (-2,1.5){1};
	\draw [shorten >=.1cm,shorten <=.1cm,->]([yshift=2.5pt]{pic cs:c1}) -- ([yshift=2.5pt]{pic cs:c2});
    \end{tikzpicture}
\label{tb:ss_leray_1}
\end{table}

This implies that 
$H^\bullet(C_{1,1};\mathbf{Q})$
is $\mathbf{Q}$ in degree $0,$ $\mathbf{Q}(-1)$ in degree $2,$ $\mathbf{Q}(-2)$ in degree $3$ and $0$ in all other degrees. 
\printbibliography
\end{document}